\documentclass[a4paper,11pt,headsepline]{article}
\usepackage[english]{babel}
\usepackage[T1]{fontenc}
\usepackage[utf8]{inputenc}
\usepackage[autostyle,german=quotes]{csquotes}

\usepackage{setspace}
\usepackage[left=3.5cm,right=2.5cm,top=1cm,bottom=3cm,includeheadfoot]{geometry}

\usepackage{graphics}
\usepackage{graphicx}
\usepackage{color}
\usepackage{subfigure} 
\usepackage{caption}
\usepackage{pgf}
\usepackage{tikz}
\usetikzlibrary{shapes,arrows}
\usepackage{amsmath}
\usepackage{shadethm}
\usepackage{amsfonts}
\usepackage[mathscr]{eucal}
\usepackage{dsfont}
\usepackage{amssymb}
\usepackage{amsthm}
\usepackage{nicefrac}

\usepackage{multicol}
\usepackage{ifthen}
\usepackage{calc}
\usepackage{nomencl}
\usepackage{scrpage2}

\theoremstyle{plain}% default
\newtheorem{thm}{Theorem}[section]
\newtheorem{lem}[thm]{Lemma}
\newtheorem{prop}[thm]{Proposition}
\newtheorem{cor}[thm]{Corollary}

\theoremstyle{definition}
\newtheorem{defn}[thm]{Definition}

\newtheoremstyle{remark}% name
{15pt}% Space above
{1em}% Space below 
{}% Body font
{}% Indent amount: Indent amount: empty = no indent, \parindent = normal paragraph indent
{\bf}% Theorem head font
{:}% Punctuation after theorem head
{ }% Space after theorem head: Space after theorem head: { } = normal interword space; \newline = linebreak
{}% Theorem head spec (can be left empty, meaning `normal')

\theoremstyle{remark}

\newtheoremstyle{not}% name
{15pt}% Space above
{5pt}% Space below 
{}% Body font
{}% Indent amount: Indent amount: empty = no indent, \parindent = normal paragraph indent
{\bf}% Theorem head font
{:\hspace{0.05cm}}% Punctuation after theorem head
{ }% Space after theorem head: Space after theorem head: { } = normal interword space; \newline = linebreak
{}% Theorem head spec (can be left empty, meaning `normal')

\theoremstyle{not}

\newtheoremstyle{proof}% name
{}% Space above
{2pt}% Space below 
{}% Body font
{}% Indent amount: Indent amount: empty = no indent, \parindent = normal paragraph indent
{\it}% Theorem head font
{:}% Punctuation after theorem head
{ }% Space after theorem head: Space after theorem head: { } = normal interword space; \newline = linebreak
{}% Theorem head spec (can be left empty, meaning `normal')

\theoremstyle{proof}

\renewcommand{\O}{\mathcal{O}}
\newcommand{\Os}{\tilde{\O}}
\newcommand{\Uad}{U^{ad}}

\newcommand{\On}{\Omega_{n}}

\newcommand{\Oext}{\Omega^{ext}}

\newcommand{\RO}{\partial \Omega}

\newcommand{\tr}{\text{tr}}

\newcommand{\Umg}[2]{B_{#1}(#2)}

\newcommand{\R}[1]{\mathbb{R}^{#1}}
\newcommand{\N}[1]{\mathbb{N}^{#1}}
\newcommand{\C}[2]{C^{#1}(#2)}
\newcommand{\Ck}[3]{[C^{#1}(#2)]^{#3}}		% C^{k}(Omega)

\newcommand{\D}[2]{\mathscr{D}^{#1}(#2)}
\newcommand{\G}{\mathcal{G}}

\newcommand{\seq}[1]{(#1)_{n\in \mathbb{N}}}
\newcommand{\subseq}[1]{(#1)_{k \in \mathbb{N}}}

\newcommand{\Norm}[2]{\Vert #1 \Vert_{#2}} 
\newcommand{\dist}{\text{dist}}

\begin{document}
% front matter

\title{\sc Modeling, Minimizing and Managing the Risk of Fatigue for Mechanical Components}

\author{{\sc L. Bittner${}^*$, H. Gottschalk${}^*$, M. Gr\"oger${}^{*,\dag}$}\\
{\sc N. Moch${}^*$, M. Saadi${}^*$ and S. Schmitz${}^\ddag$}\\
{\small  ${}^ *$School of
Mathematics and Science, Gau\ss str. 20,} \\\small{ Bergische Universit\"at 
Wuppertal, Germany}\\
{\small ${}^\dag$Siemens Energy, Service Engineering, Mellinghofer Str. 55, M\"ulheim/Ruhr, Germany}\\
{\small ${}^\ddag$Siemens Energy, Gas Turbine Engineering, Huttenstr. 12,  Berlin, Germany}\\
{\small \tt laura.bittner@math.uni-wuppertal.de}\\
{\small \tt \{hanno.gottschalk,nmoch\}@uni-wuppertal.de}\\
{\small \tt \{michael.groeger.ext,schmitz.sebastian\}@siemens.com}\\
}

\maketitle

\vspace{.3cm}

\begin{abstract}
 Mechanical components that are exposed to cyclic mechanical loading fail at loads that are well below the ultimate tensile strength. This process is known as fatigue. The failure time, that is the time when a first crack forms, is highly random. In this work we review some recent developments in the modelling of probabilistic failure times, understood as the time to the formation of a fatigue crack.

We also discuss the how probabilistic models can be used in shape design with the design intent of optimizing the component's reliability. We give review a recent existence result for optimal shapes and we discuss continuous and discrete shape derivatives. Another application is optimal service scheduling. The mathematical fields involved range from reliability statistics over stochastic point processes, multiscale modeling, PDEs on variable geometries, shape optimization and numerical analysis to operations research.   
\end{abstract}



%\tableofcontents

\section{\label{sec:Intro}Introduction}

If the mechanical failure of a component would be predictable, everybody would use a mechanical component just until the end of this predicted life time. It would fall apart during the very next use cycle due to a crack which develops exactly at its weakest point. There would be no safety issues related to  reliability, because everybody knows, when 'it' will happen. At this point it is already clear that this is not the actual state of our world. Chance plays a major role in whether or not a mechanical component can survive a predefined load history. Probabilistic models of failure times thus provide a more realistic description of reality, than deterministic life prediction does.  

In this work, we review a model of probabilistic life prediction for mechanical components under cyclic loading that has recently been developed \cite{Hanno,SBKRSG,Schmitz}, tested \cite{SBKRSG}, numerically implemented and applied to gas turbine engineering \cite{SSGBRK,SSBGRK2}. We also review the foundations of the model in materials science and give an insight into the microscopic origins of the scatter in life time. As these origins heavily rely on the material and the damage mechanism considered, we here restrict to polycrystalline metal and low cycle fatigue (LCF).   Note that, even under controlled lab conditions, the scatter in LCF life -- defined as the number of load cycles to crack initiation -- is about one order of magnitude \cite{Erm}.

Probabilistic models of failure have been studied since the pioneering work of Weibull \cite{Wei}. In the context of ceramics, the probabilistic approach has become widely used \cite{Fett}. For metals and LCF the approach proposed by two of the authors in collaboration with material scientists and gas turbine engineers is new \cite{Hanno,SSBGRK2}, for an alternative approach see \cite{Vorm}.  We review the probabilistic model for LCF in Section 2.1 from a general prospective based on point processes and their associated first failure times. 

Interestingly, the probabilistic model for LCF, but also the ceramics based models \cite{BGS}, have a number of interesting implications. Choosing the form of a component can be seen as the choice of a volume $\Omega\subseteq \R{3}$ that is filled with material.  If we are able to associate a probability law for the failure time distribution for each admissible shape $\Omega$, we can ask, which of the shapes leads to the most reliable component. It turns out that there are at least three different notions \cite{BG}, what optimal reliability could actually mean: A  vendor might want to minimize the warranty cost until the warranty time $t^*$ and might not care, if the components falls apart shortly afterwards. Contrary to such design to life optimizations, a vendor who is more interested in the customer's benefit might like to optimize the reliability up to any given time. This would mean that the probability of failure for her or his components is less than for any other design alternative. Such an approach however does, at least theoretically, not exclude designs where the risk of failure is concentrated in short periods. During the short periods of elevated risk people might be exposed to unethical hazards. Optimal reliability might thus be also understood as the minimization of instantaneous hazard at any time. We review these notions in Section \ref{sec:Basics}. Fortunately we are able to identify some situations, where the three given notions of optimal reliability do coincide.  In particular this is true for proportional hazard models related to the Cox-process, one of the work horses of survival analysis \cite{EscobarMeeker}.

Section \ref{sec:ProbMod} is closed with the introduction of the local probabilistic model for LCF in Section 2.3. Besides the original probabilistic local Weibull model \cite{Hanno,SSBGRK2,Schmitz}, we here propose a new and unpublished variant based on Gompert's law of exponential hazard. While the deterministic life function is used as a scale variable in the Weibull model, it is a location variable in the Gompert's for LCF. We also show that both models fall under the proportional hazard paradigm and thus can not be used for design to life activities. For the Weibull model, this has already been observed in \cite{BG}.    

In Section \ref{sec:FEA} we discuss the actual calculation of failure probabilities for complex geometries using finite elements. As a real world application, we provide the example of a 3D Turbo charger.

Section \ref{sec:MultiScale} gives insight in ongoing research on micro mechanical models for the probabilistic life calculation. In particular we present some new numerical results on the distribution of Schmid factors in the case of muliaxial stress, see \cite{GSSKRB} for some prior results for uniaxial stress states. We also give a brief outline, how these micro models can be integrated in the macroscopic probabilistic life time description.    

 In the following we deal with the quest of optimal reliability. Section \ref{sec:OpShapes} translates the problem of optimal reliability to a problem of shape optimization \cite{ShapeOpt,BucBut05,SokZol92}. So the first question is the existence of optimal shapes. A number of results have been obtained in this direction \cite{Chen75,Fujii,Epp07,Eppler_Unger97,ShapeOpt,BGS}, but the optimal reliability has a quite singular shape functional and therefore does not fit into the existing framework of weak $H^ 1$-solutions. We therefore follow the approach in \cite{Hanno,BG,Schmitz} and use elliptic regularity theory in order to prove compactness results on the graph of the optimization  problem that are needed in the existence proof \cite{ShapeOpt}.

In Section 5 we give an outlook on the theory of continuous shape derivatives for objective functionals \cite{SokZol92,ShapeOpt,Schmitz} that stem from our probabilistic fatigue models. Despite the rather singular nature of the objective functionals, one can once again apply elliptic regularity to prove existence. However, the adjoint equation has some surprising features and has to be interpreted in a distributional sense. We also give a short comment on first order optimality conditions, see also \cite{BGS,Schmitz}.

Another application of the probabilistic models for LCF is given in Section \ref{sec:RiskManagement}, where, based on the knowledge of fatigue risk figures and economic impact, service plans are valutated economically and optimal service intervals are determined.

Finally, we draw some conclusions and give an outlook to future research work in Section \ref{sec:Conclusion}.

\section{\label{sec:ProbMod} Probabilistic Models for Fatigue Crack Initiation}

\subsection{Probabilistic Failure Modeled with Point Processes}
Let $\Omega\subseteq \R{3}$ be a bounded, open domain that represents the shape of a mechanical component, i.e. the region filled with matter and let $\partial\Omega=\overline{\Omega}\setminus \Omega$ be its boundary. We consider a crack as an event that happens at some time $t\in\R{+}=[0,\infty)$ in some location $x\in\overline{\Omega}$. The configuration space for crack initiations thus is $\mathcal{C}=\overline{\Omega}\times\R{+}$. By $\mathcal{R}(\mathcal{C})$ we denote the set of Radon measures on $\mathcal{C}$, i.e. the set of $\sigma$-finite measures such that compact sets have finite measure. A radon measure $\rho$ is atom free, if $\rho(\{c\})=0$ for all $c\in\mathcal{C}$. 

Let $\mathcal{R}_c(\mathcal{C})\subset \mathcal{R}(\mathcal{C})$ the counting measures on $\mathcal{C}$ that associate natural numbers to measurable regions in $\mathcal{C}$. We note that for $\gamma\in \mathcal{R}_c(\mathcal{C})$, there exists a unique representation $\gamma\restriction_{\mathcal{C}_t}=\sum_{j=1}^nb_j\delta_{c_j}$, with $c_j=(x_j,t_j)\in \mathcal{C}_t=\overline{\Omega}\times[0,t]$, $b_j\in\mathbb{N}$. $\gamma$ is called simple if $b_j=1$ for all $j=1,\ldots,n$ and all $t$. Here $\delta_c$ is the Dirac measure with mass one in $c\in\mathcal{C}$.

\begin{defn}\label{def:CIHistory}\textbf{(Crack Initiation History)}
\begin{itemize}
\item[(i)] A crack initiation history is a simple radon counting measure $\gamma \in\mathcal{R}_c(\mathcal{C})$. $\gamma(C)\in\mathbb{N}_0$ stands for the number of cracks initiated in some measurable set of configurations $C\subseteq\mathcal{C}$.
\item[(ii)] $\tau(\gamma)=\inf\{ t\geq 0:\gamma(\mathcal{C}_t)>0\}$ is the failure time associated with the crack initiation history $\gamma$.  
\end{itemize}

\end{defn}

In general, the formation of a crack is a random event. Thus also crack initiation histories have to be random. This connects technical failure with point processes:

\begin{defn} \label{def:PP}\textbf{(Point Processes)}
\begin{itemize}
\item[(i)] Let $(\mathscr{X},\mathscr{A},P)$ be a probability space and let $\mathcal{R}_c(\mathcal{C})$ be endowed with the sigma algebra $\mathcal{B}$ generated by $\gamma \to \int_{\mathcal{C}}f\,d\gamma$, where $f\in C_0(\mathcal{C})$.  Then a measurable map $\gamma:\Omega\to \mathcal{R}_c(\mathcal{C})$ is called a point process.
\item[(ii)] A point process is called simple, if its realizations $\gamma(\omega)$ are simple $P$ almost surely.
\item[(iii)] A point process is non-atomic, if for all $c\in \mathcal{C}$, $P(\gamma(\{c\})>0)=0$ holds.
\item[(iv)] A point process $\gamma$ has independent increments, if, for all measurable and mutually disjoint sets of configurations $C_1,\ldots,C_n\subseteq \mathcal{C}$, $\gamma(C_1),\ldots,\gamma(C_n):\mathscr{X}\to\mathbb{N}_0\cup\{\infty\}$ are independent random variables.  
\end{itemize}

\end{defn} 

For a crack initiation process, that is a random crack initiation history, the above properties (ii) and (iii) immediately make sense: Simplicity corresponds to the fact that no two cracks can initiate at one place at the same time, whereas non atomic crack processes do not have distinguished points and times in the continuum, where cracks nucleate with  probability larger than zero.  Property (iv) is a reasonable assumption, if one considers the initial phase where cracks have not grown to a size where they significantly influence each other.

\begin{defn} \textbf{(Crack Initiation Process)}\\
A crack initiation process is a simple, non atomic point process on ${\cal C}_t=\overline{\Omega}\times\R{+}$. The associated time of first failure is $\tau=\tau(\gamma)$.
\end{defn}

The following theorem is due to S. Watanabe \cite{Watanabe}, see also \cite{Kallenberg}.

\begin{prop} \label{prop:classPPP}\textbf{(Classification of the Poisson Point Process)}\\
A simple, non atomic point process with independent increments $\gamma$ is a Poisson point process (PPP). That is, there exists a atom free Radon measure $\rho\in \mathcal{R}(\mathcal{C})$, called the intensity measure, such that $P(\gamma(C)=n)=e^{-\rho(C)}\frac{\rho(C)^n}{n!}$ for all measurable $C\subseteq\mathcal{C}$.

Conversely, a PPP with a non atomic intensity measure $\rho$ is non atomic, simple and has independent increments.
\end{prop} 

This implies the following:

\begin{cor}\label{cor:FailureTimeDistribution}\textbf{(First Failure Time Distribution)}\\
Let $\gamma$ be a crack initiation process with independent increments. 

\begin{itemize}
\item[(i)] Then, the cumulative distribution function of the first failure time $\tau=\tau(\gamma)$ is
$
F_\tau(t)=1-e^ {-\rho(\mathcal{C}_t)}$, $t\in\R{+}$,
where $\rho$ is the intensity measure of the $PPP$ $\gamma$.
\item[(ii)] Suppose that $H(t)=\rho(\mathcal{C}_t)$ is differentiable, then $h(t)=H'(t)$ is the Hazard rate, $h(t)=\lim_{\Delta\searrow 0}\frac{1}{\Delta}P(\tau\in [t,t+\Delta]|\tau>t)$.
\end{itemize}
\end{cor}
\begin{proof}
(i) By Prop. \ref{prop:classPPP}, $\gamma$ is a PPP with intensity measure $\rho\in\mathcal{R}(\mathcal{C})$. The event $\{\tau\leq t\}$ is equal to $\{\gamma(\mathcal{C}_t)>0\}$ and thus $F_\tau(t)=P(\tau\leq t)= P(\gamma(\mathcal{C}_t)>0)=1-P(\gamma(\mathcal{C}_t)=0)=1-e^ {-\rho(\mathcal{C}_t)}$.

(ii) This follows from the representation of the survival function $S_\tau(t)=1-F_\tau(t)=e^{-\int_0^th(\tau)d\tau}$ where $h(t)$ is the Hazard rate, see \cite{EscobarMeeker}.  
\end{proof}

\subsection{Optimal Reliability}

We now pass on to the problem of design, which is the choice of one form $\Omega$ from a set of design alternatives $\mathcal{O}$, called the set of admissible shapes. As the avoidance of failure is one important design objective, we can ask for an optimal design with respect to this design criterion:

\begin{defn} \label{def:OptRel} \textbf{(Optimal Reliability)}\\
Let $\mathcal{O}$ be some set of admissible shapes $\Omega \subset \R{3}$. Suppose that for $\Omega\in\mathcal{O}$ there is a crack initiation process $\gamma_\Omega:(\Xi,\mathcal{A},P)\to (\mathcal{R}_c,\mathcal{B})$ with associated first failure time $\tau_\Omega$. Then there exist several options to formulate the problem of optimal reliability:
\begin{itemize}
\item[(i)] $\Omega^*\in\mathcal{O}$ fulfills optimal reliability at time $t\in(0,\infty)$, if $F_{\tau_{\Omega^*}}(t)\leq F_{\tau_{\Omega}}(t)$ for all $\Omega\in\mathcal{O}$.  
\item[(ii)] $\Omega^*\in\mathcal{O}$ fulfills optimal reliability in first stochastic order,  if $F_{\tau_{\Omega^*}}(t)\leq F_{\tau_{\Omega}}(t)$ for all $\Omega\in\mathcal{O}$ and $t\in\R{+}$.
\item[(iii)]  $\Omega^*\in\mathcal{O}$ fulfills optimal reliability in terms of hazard,  if the hazard rate $h_{\tau_{\Omega}}(t)$ exists for all $\Omega\in\mathcal{O}$ and $h_{\tau_{\Omega^*}}(t)\leq h_{\tau_{\Omega}}(t)$ for all $\Omega\in\mathcal{O}$ and $t\in\R{+}$. 
\end{itemize}
\end{defn}
As one easily sees, in Def. \ref{def:OptRel} we have (iii) $\Rightarrow$(ii) $\Rightarrow$ (i). However, in some situations also the opposite implication holds. We then speak of a no design to life situation, since then it is not possible to optimize reliability with respect to some warranty time $t$.  The weak no design to life condition is the equivalence of (i) and (ii). This means, that for any time horizon $t>0$ it is more probable that design $\Omega$ fails until $t$ than the failure of $\Omega^*$, provided this is true for some time $t^*$ (e.g. the warranty time). The equivalence of (iii) and (i) even means that, provided the operation of $\Omega^*$ over the time span $(0,t^*)$ is more reliable than the operation  of $\Omega$, then $\Omega^*$ will produce an failure with a lesser likelihood than $\Omega$ at and instant of time $t>0$. In other words $\Omega^*$ is always safer than $\Omega$, provided this is true over some time span.  

\begin{prop} \label{prop:NoDesignToLife} \textbf{(Sufficient Conditions for No Design To Life)}\\
Let a failure time model $\tau_\Omega$ be given as described above. 
\begin{itemize}
\item[(i)] Suppose that the $\Omega$ dependence of the failure time is given by a scale variable $\eta_\Omega$ i.e. there exists a strictly monotonic distribution function $F_0:\R{+}\to[0,1]$ such that $F_{\tau_\Omega}:\R{+}\to[0,1]$ is given by $F_{\tau_\Omega}=F_0\left(\frac{t}{\eta_\Omega}\right)$.    Then Def. \ref{def:OptRel} (i) and (ii) are both equivalent to $\eta_{\Omega^*}\geq \eta_\Omega$ $\forall \Omega\in\mathcal{O}$. 
\item[(ii)] If in addition to the assumptions in (i) the hazard rates $h_0(t)$ exists for $F_0(t)$ and if $h_0(t)$ is increasing, then also condition (iii) in Def.\ \ref{def:OptRel} is equivalent to $\eta_{\Omega^*}\geq \eta_\Omega$ $\forall \Omega\in\mathcal{O}$.
\item[(iii)] Suppose that the $\Omega$ dependence is given by a variable $C_\Omega$ such that there exists a  positive baseline hazard function $h_0(t)$ of $F_0$, $t\in(0,\infty)$, where $h_{\tau_\Omega}(t)=C_\Omega h_0(t)$ holds. Then, the conditions (i)--(iii) of Def. \ref{def:OptRel} are equivalent to $J_{\Omega^*}\leq C_\Omega$ $\forall \Omega\in\mathcal{O}$.
\end{itemize}
\end{prop}
\begin{proof}
(i) This is an immediate consequence of the fact that $F_0(t)$ is strictly monotonically increasing in $t$. In fact, $F_0\left(\frac{t}{\eta_{\Omega^ *}}\right) \leq F_0\left(\frac{t}{\eta_{\Omega^ *}}\right)$ for one $t\in \R{+}$ is equivalent to $\eta_{\Omega^ *}\geq \eta_\Omega$. 

(ii) For $\Omega\in\mathcal{O}$ we get 
\begin{equation}
S_{\tau_\Omega}(t)=\exp\left\{\int_0^{t/\eta_\Omega} h_{0}(\tau)\,d\tau\right\}\mbox{ and thus } h_{\tau_\Omega}(t)=h_0\left(\frac{t}{\eta_\Omega}\right)\frac{1}{\eta_\Omega}.
\end{equation}
Now suppose that $F_{\tau_{\Omega^*}}(t)\leq F_{\tau_\Omega}(t)$, then, by the strict positivity of $h_0$ this is equivalent to $\frac{t}{\eta_{\Omega^*}}\leq \frac{t}{\eta_{\Omega}}$ which is equivalent to $\eta_{\Omega^*}\geq \eta_\Omega$.  But then, as $h_0(t)$ is monotonically increasing, 
\begin{equation}
h_{\tau_{\Omega^*}}(t)=h_0\left(\frac{t}{\eta_{\Omega^*}}\right)\frac{1}{\eta_{\Omega^*}}\leq h_0\left(\frac{t}{\eta_\Omega}\right)\frac{1}{\eta_\Omega}=h_{\tau_{\Omega}}(t)
\end{equation}
holds for all $t\geq 0$.

(iii) Suppose holds Def. \ref{def:OptRel} (i)  holds for some $t$. Then,
\begin{equation}
 J_{\Omega^*}\int_0^th_0(\tau)\,d\tau \leq J_\Omega\int_0^th_0(\tau)\,d\tau.
 \end{equation}
 As $h_0$ is positive, $J_{\Omega^*}\leq J_\Omega$ follows. But this implies that $h_{\tau_{\Omega^*}}(t)=  J_{\Omega^*}  h_0(t)\leq J_\Omega h_0(t)=h_{\tau_{\Omega}}(t)$ for arbitrary $t>0$, which is condition Def. \ref{def:OptRel} (iii).
\end{proof}

The situation described in point (iii) of Prop. \ref{prop:NoDesignToLife} is related to the proportional hazard approach of the Cox model \cite{EscobarMeeker}. 

The above assumption of a (strictly)  increasing baseline hazard rate $h_0(t)$ is essential for our conception of fatigue. If $h_0(t)$ and thus $h_{\tau_\Omega}(t)$ would not increase in time $t$, the component would gain fitness while being used. Except for short initial periods of infancy mortality, this is contrary to the general experience that technical devices break more easily, if they are aged.

\subsection{Probabilistic Models for Fatigue Cracking}

We now shortly review some models which have been discussed in \cite{Hanno,BG} in all detail.  Let $\sigma:\Omega \to \R{3\times 3}$ be a stress tensor field associated with the form $\Omega\in\mathcal{O}$. Here we model $\sigma$ in the context of linear elasticity \cite{Cia1988}. Thus, $\sigma(u)=\lambda \tr(\nabla u) I+\mu(\nabla u+\nabla u^T)$ for the L\'ame constants $\lambda,\mu>0$ and $I$ the $3\times 3$ unit matrix. $u=u(\Omega)$ is the displacement field, which solves the elliptic system of PDEs
\begin{align}
\label{eqa:ElasticityStrong}
\begin{split}
-\nabla \cdot \sigma(u)=f(\Omega) & \mbox{ on }\Omega\\
\sigma(u)\nu=g(\Omega) & \mbox{ on }\partial \Omega_N\\
u=0 & \mbox{ on }\partial \Omega_D.\\ 
\end{split}
\end{align} 
Here $f_\Omega:\Omega\to \R{3}$ is the volume force density and $g_\Omega:\partial\Omega\to \R{3}$ is the surface force density. $\partial\Omega$ is decomposed into a Dirichlet part with non zero surface volume $\partial\Omega_D$ and a part with natural boundary conditions $\Omega_N$. $\nu:\partial\Omega\to\R{3}$ is the outward normal vector field. For the existence and uniqueness of (weak) solutions see \cite{Cia1988}.

Based on a weak solution $u\in [H^1(\Omega)]^3$ and the associated stress field $\sigma(u)$, one considers the situation where the loads switch from $f_{1,\Omega}$ and $g_{1,\Omega}$ to $f_{2,\Omega}$ and $g_{2,\Omega}$ periodically. 
This of course also leads to a cyclic behavior of the stress tensor $\sigma$ that oscillates  between $\sigma_1=\sigma(u_1)$ and $\sigma_2=\sigma(u_2)$.  The stress amplitude of this cyclic motion is $\sigma_a=\frac{1}{2}(\sigma(u_1)-\sigma(u_2))$. Note that by the linearity of \eqref{eqa:ElasticityStrong}, $\sigma_a=\frac{1}{2}\sigma(u)$ with $u$ solves \eqref{eqa:ElasticityStrong} for $f_\Omega=f_{1,\Omega}-f_{2,\Omega}$ and $g_\Omega=g_{1,\Omega}-g_{2,\Omega}$.

Let $\sigma'_a=\sigma_a-\frac{1}{3}\tr(\sigma_a)I$ be the trace free part of $\sigma_a$. Then, the eleastic von Mises amplitude  stress is defined by $\sigma_{a}^{\rm el}=\left(\frac{3}{2}\sigma_a':\sigma_a'\right)^{\frac{1}{2}}$, where $a:b=\sum_{i,j=1}^3a_{i,j}b_{i,j}$ is a scalar quantity for $a,b\in\R{3\times 3}$. 

When undergoing a load cycle described by a elastic-plastic stress amplitude $\sigma_a^{\rm el-pl}$, the material locally travels through a displays a stress strain hysteresis \ref{fig:StressStrain}, as displayed in Figure \ref{fig:StressStrain} (a).

\begin{figure}[t]
\includegraphics[width=.35\textwidth]{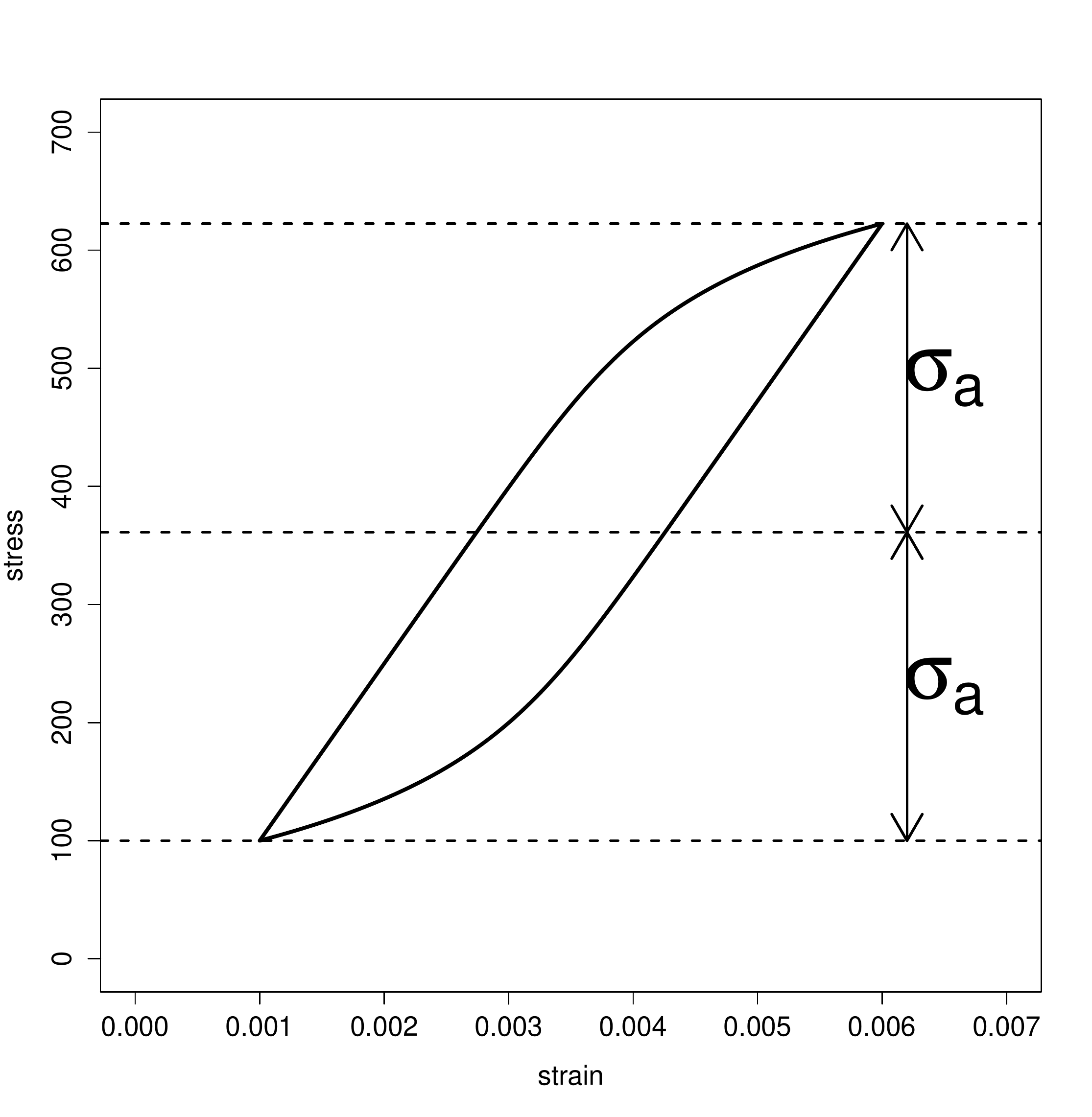} \includegraphics[width=.65\textwidth]{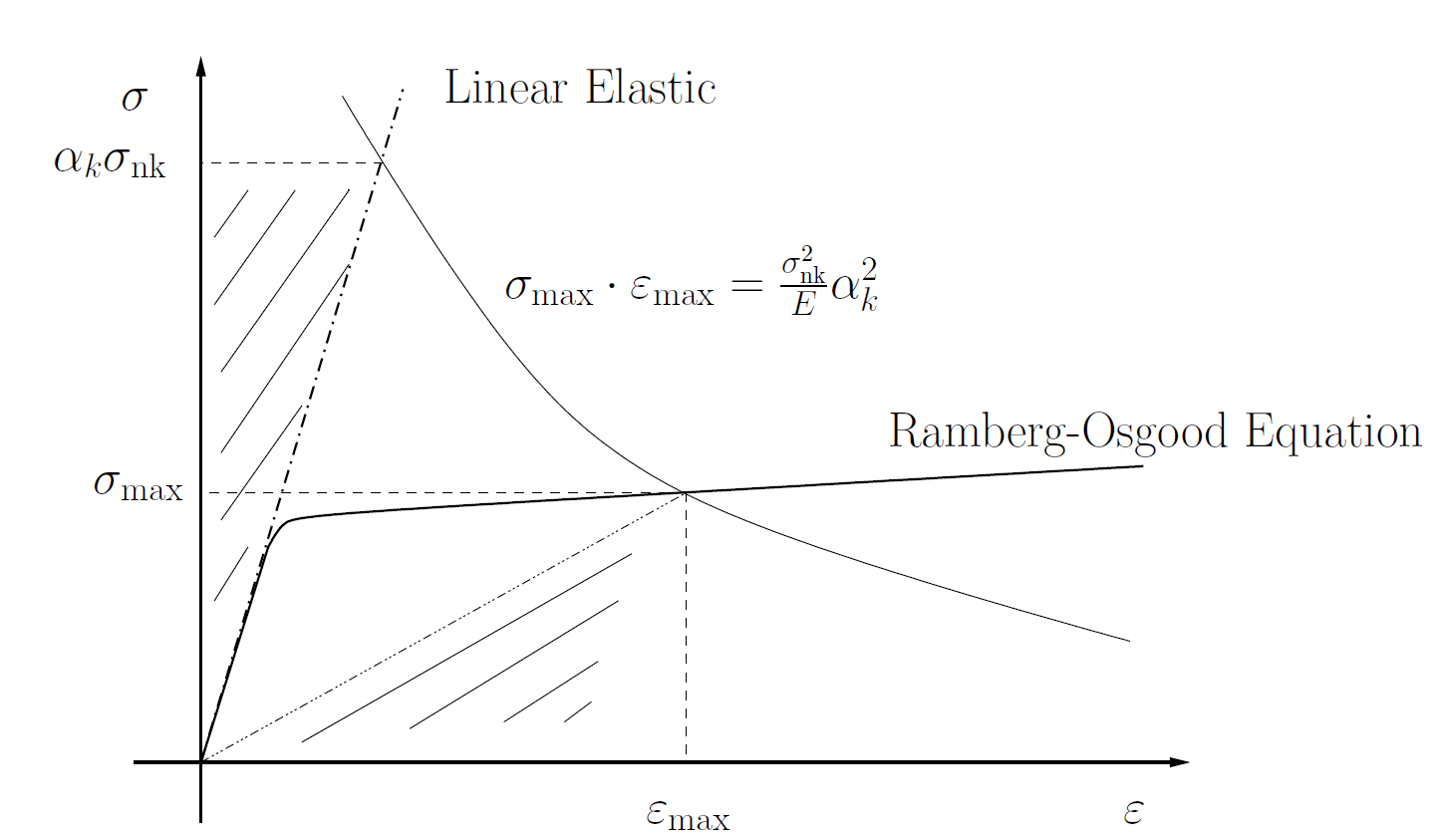}
\caption{(a) Stress Strain Hysteresis and (b) Neuber shakedown -- Notch factors $\alpha_k=1$ holds for solutions of the elasticity PDE}
\label{fig:StressStrain}
\end{figure} 

The limiting cycle is described by the Ramberg-Osgood equation
\begin{equation} \label{eqa:Ramberg-Osgood}
\varepsilon_a^{\rm el-pl}={\rm RO}(\sigma_a^{\rm el-pl})=\frac{\sigma_a^{\rm el-pl}}{E}+\left(\frac{\sigma_a^{\rm el-pl}}{K}\right)^{\frac{1}{n'}}.
\end{equation} 
Here $E$ is Young's modulus, $K$ is the hardening coefficient and $n'$ is the hardening exponent \cite{Werkstoffe,Erm}.

One of the problems with the linear elasticity equation \eqref{eqa:ElasticityStrong} is, that the elastic von Mises stress amplitude $\sigma_a^{\rm el}$ largely overestimates the actual elastic-plastic stress amplitude $\sigma_a^{\rm el-pl}$ as it does not take into account plastic yielding causing stress relaxation. A procedure often applied to convert $\sigma_a^{\rm el}$ into $\sigma_a^{\rm el-pl}$ is the Neuber shake down
\begin{equation}
\frac{\left({\sigma_a}^{\rm el}\right)^2}{E}=\frac{\left({\sigma_a}^{\rm {el-pl}}\right)^2}{E}+\sigma_a^{\rm el-pl}\left(\frac{\sigma_a^{\rm el-pl}}{K}\right)^{\frac{1}{n'}},
\end{equation}
which has to be solved for $\sigma_a^{\rm el-pl}={\rm SD}(\sigma_a^{\rm el})$.   Entering $\sigma_a^{\rm el-pl}$ into the Ramberg-Osgood equation, one obtains $\varepsilon_a^{\rm el-pl}$. Finally this quantity is related to the deterministic crack initiation time $N_{i_{\rm det}}$ via the Coffin--Manson-Basquin equation
\begin{equation}
\label{eqa:CMB}
\varepsilon_a^ {\rm el-pl}={\rm CMB}(N_{i_{\rm det}})=\frac{\sigma_f'}{E}\left(2N_{i_{\rm det}}\right)^b+\varepsilon_f'\left(2N_{i_{\rm det}}\right)^ c,
\end{equation}  
with $\sigma_f',\varepsilon_f'>0$ and $b,c<0$ material constants \cite{Werkstoffe,Vorm}. This equation can now be inverted to obtain $N_{i_{\rm det}}$, the deterministic number of load cycles to crack initiation. This number is interpreted as a deterministic first failure time. For the first failure time of the entire component $\Omega$, we then obtain
\begin{equation}
\label{eqa:Ndet}
N_{i_{\rm det}}(\Omega)=\inf_{x\in\partial\Omega} N_{i_{\rm det}}\left(\varepsilon_a^ {\rm el-pl}(x)\right)
\end{equation}
Note that fatigue cracks always take their origin at the surface of the component. Therefore the infimum in \eqref{eqa:Ndet} is only over all surface points in $\partial \Omega$.

In the next step we use the analysis of the previous subsection for setting up probabilistic crack initiation models. In order to keep the structural properties of the probabilistic model as close to the deterministic life prediction as possible, the idea is to use $N_{i_{\rm det}}(\varepsilon_a^{\rm el-pl}(x))$ either as a local scale variable or a local location variable to define the density of the intensity measure $\rho$ that governs the law of the PPP crack initiation process. It turns out that the following two models can be set up in order to fulfil the requirements of no design to life, cf. Prop. \ref{prop:NoDesignToLife}:

\begin{defn} \label{def:FailureTimeModels}\textbf{(Local Crack Initiation Models)}
\begin{itemize}
\item[(i)] The local, probabilistic Weibull model is given by the intensity measure $\rho$ on $\mathcal{C}$ defined by the density function $\varrho(x,t)$ with respect to the surface measure $dA$ on $\partial\Omega$ and the Lebesgue measure on $\R{+}$
\begin{equation}
\label{eqa:locW}
\varrho_\Omega(x,t)=\frac{\bar m}{N_{i_{\rm det}}(\varepsilon_a^{\rm el-pl}(x))}\left(\frac{t}{N_{i_{\rm det}}(\varepsilon_a^{\rm el-pl}(x))}\right)^ {\bar m-1},~~m\geq 1.
\end{equation}
\item[(ii)] The local, probabilistic Gompert's model is given by the intensity measure $\rho$ on $\mathcal{C}$ defined by the density function $\varrho_\Omega(x,t)$ with respect to the surface measure $dA$ on $\partial\Omega$ and the Lebesgue measure on $\R{+}$
\begin{equation}
\label{eqa:locG}
\varrho_\Omega(x,t)=C\exp\left\{\alpha\left(t- N_{i_{\rm det}}(\varepsilon_a^{\rm el-pl}(x))\right)\right\},~~\alpha,C>0.
\end{equation}
\end{itemize}
\end{defn}

We note that both models have to be calibrated with experimental data, see e.g. \cite{SSBGRK2,Schmitz} for the maximum likelihood calibration of the local Weibull model. Note that also the CMB parameters $\varepsilon_f'$ and $\varsigma_f'$ are being recalibrated in this procedure. The local Gompert's model has the disadvantage that a (small) crack initiation hazard is already present at $t=0$.

Let us recall that the Weibull distribution ${\rm Wei}(\eta,\bar m)$ has the survival function $S(t)=e^{-\left(\frac{t}{\eta}\right)^{\bar m}}$ while the Gomperts distribution ${\rm Gomp}(C,\alpha)$ has survival function $S(t)=e^{-C(e^{\alpha t}-1)}$. The following proposition summarizes the properties of the associated first failure times:

\begin{prop}
\label{prop:WeibullAndGomperts} \textbf{(Properties of the Local Weibull and Gompert's Model)}
\begin{itemize}
\item[(i)] Let $\tau_\Omega$ be the first failure time of the local Weibull model. Then, $\tau_\Omega$ is ${\rm Wei}(\eta_\Omega,\bar m)$-distributed with $m$ from \eqref{eqa:locW} and 
\begin{equation}
\eta_\Omega=\left(\int_{\partial\Omega}\left(\frac{1}{N_{i_{\rm det}}(\varepsilon_a^{\rm el-pl}(x))}\right)^ {\bar m} \,dA\right)^{-\frac{1}{\bar m}}.
\end{equation}
In particular, the strong no design to life property from Prop. \ref{prop:NoDesignToLife} (ii) applies.
\item[(ii)] Let $\tau_\Omega$ be the first failure time of the local Gompert's model. Then, $\tau_\Omega$ is ${\rm Gomp}(J_\Omega,\alpha)$-distributed with $\alpha$ from \eqref{eqa:locG} and 
\begin{equation}
J_\Omega=\frac{C}{\alpha}\int_{\partial\Omega}e^{-\alpha N_{i_{\rm det}}(\varepsilon_a^{\rm el-pl}(x))} \,dA.
\end{equation}
In particular, the strong no design to life property from Prop. \ref{prop:NoDesignToLife} (iii) applies.

\end{itemize}
\end{prop} 
\begin{proof}
(i) We have $S_{\tau_\Omega}(t)=e^{\rho(\mathcal{C}_t)}$ with $\mathcal{C}_t=\overline{\Omega}\times [0,t]=\partial\Omega\times [0,t]\cup\Omega\times [0,t]$. As $\rho$ is absolutely continuous with respect to the surface measure $dA\times dt$ concentrated on $\partial\Omega\times [0,t]$, $\rho(\mathcal{C}_t)=\rho(\partial\Omega\times[0,t])=\int_0^t\int_{\partial\Omega} \varrho_\Omega(x,\tau)\, dAd\tau$ follows. We thus get 
\begin{align}
\begin{split}
\rho(\mathcal{C}_t)&=\int_0^t\int_{\partial\Omega} \frac{\bar m}{N_{i_{\rm det}}(\varepsilon_a^{\rm el-pl}(x))}\left(\frac{\tau}{N_{i_{\rm det}}(\varepsilon_a^{\rm el-pl}(x))}\right)^ {\bar m-1} dAd\tau\\
&= t^{\bar m}\int_{\partial\Omega}\left(\frac{t}{N_{i_{\rm det}}(\varepsilon_a^{\rm el-pl}(x))}\right)^ {\bar m} dA=\left(\frac{t}{\eta_\Omega}\right)^{\bar m}.
\end{split}
\end{align}
(ii) In analogy to (ii) we obtain
\begin{align}
\begin{split}
\rho(\mathcal{C}_t)&=\int_0^t\int_{\partial\Omega} C\exp\left\{\alpha\left(\tau- N_{i_{\rm det}}(\varepsilon_a^{\rm el-pl}(x))\right)\right\}  dAd\tau\\
&= \frac{C}{\alpha}\left(e^{\alpha t}-1\right)\int_{\partial\Omega}e^{-\alpha N_{i_{\rm det}}(\varepsilon_a^{\rm el-pl}(x))} dA=J_\Omega \left(e^{\alpha t}-1\right).
\end{split}
\end{align}
\end{proof}

We note that also the local Weibull model can be considered as a within the framework of proportional hazard models with 
\begin{equation}
\label{eqa:propHazardWei}
J_\Omega=\int_{\partial\Omega}\left(\frac{1}{N_{i_{\rm det}}(\varepsilon^{\rm el-pl}_a(x))}\right)^{\bar m}dA.
\end{equation}

\section{\label{sec:FEA} Discretization with Finite Elements}
\label{sec:DIS}

\subsection{Discretization of the Elasticity PDE}

In this section we discuss the numerical calculation of failure probabilities for components with complex geometry, such that no analytic solution for $u$ and $\sigma$ is available. We follow the standard approach of finite element discretization. We first recall some fundamentals of finite element analysis (FEA) for the elasticity PDE.

%To fix the notation for later purposes, we introduce the standard discretization of the PDE \eqref{eqa:WeakPDE} with Lagrange finite  elements %\cite{EG,Bra}.
A finite element is defined as a triple $\{K,P(K),\Sigma(K)\}$, where
$ K \subset \R{3}$ is a compact, connected Lipschitz set with nonempty
interior called \textbf{element domain},
$P(K)$ is a finite-dimensional vector space of functions (mostly polynomials) and
  the set $\Sigma(K)=\{\varphi_1,...,\varphi_{n_{sh}}\}$ are linear forms $\phi_l:P(K)\rightarrow \R{}$ such that this space is a basis for $\mathcal{L}(P(K);\R)$ (the bounded linear functionals on $P(K)$).
Note that the linear forms $\{\varphi_1,...,\varphi_{n_{sh}}\}$ are called the local degrees of 
freedom.
The basis functions $\{\theta_1,...,\theta_{n_{sh}}\}$ in $P(K)$ which satisfies $\varphi_i(\theta_j)=\delta_{ij}$ for $1\leq i,j \leq n_{sh}$
are called local shape functions.
We call $\{K,P(K),\Sigma(K)\}$ a Lagrange finite element, if there is a set of points $\{X_{1}^K,\ldots,X_{n_{\rm sh}}^K\}\in K$ such that, for all $p \in P(K)$, $\varphi_i(p)=p(X_{i}^K)$, $1\leq i \leq n_{sh}$.

A mesh $\mathcal{T}_h$ is defined as a union of compact, connected, Lipschitz sets $ K_m $ with $\mathring{K}_m\neq \emptyset$ such that
$\{K_{m}\}_{1\leq m\leq N_{el}}$ forms a partition of $\Omega$.
A mesh can be generated from a reference element $\hat{K}$ and a set of geometric transformations $T_K : \hat{K} \rightarrow K $, which
map $\hat{K}$ to the current mesh element. We assume that the transformation $T_K$ is bijective and all mesh element are generated from
the same reference element.
By setting $\widehat{P}=P\circ T_K$ , $\widehat\theta_{j}=\theta_{j}\circ T_K$ and $\widehat\varphi_{j}(p\circ T_K)=\varphi_{j}(p)$ for $j\in\{1,\ldots,n_{\rm sh}\}$ we have:
\begin{equation}
\label{eqa:Transformation}
T_K(\hat\xi)=T_K(\hat\xi,X)=\sum_{j=1}^{n_{\rm sh}}\widehat\theta_j(\widehat{\xi})X_{j}^K,~~\widehat\xi\in
\widehat{K}.
\end{equation}

We define $\{X_1,...,X_N\}=\bigcup_{K\in \mathcal{T}_h}\{X_{1}^K,...,X_{n_{sh}}^K\}$ as the set of all the Lagrange nodes.
 For $K\in \mathcal{T}_h$ and $m\in\{1,...,n_{sh}\}$, let
\begin{align}
\begin{split}
\mathcal{T}_h\times\{1,\ldots,n_{\rm sh}\}\to&\{1,\ldots,N\}\\
                         (K,m)\to &\widehat{j}(K,m)
\end{split}
\end{align}
be the mapping which allow us to know the relation
between local and global index of nodes.
Let $\{\varphi_1,...,\varphi_N\}$ be a set of functions defined as:
\begin{align}
\begin{split}
\varphi_{i|K}(X_{m}^K)=
\left\{
  \begin{array}{ll}
    \delta_{mn}, & \hbox{if there is $n\in \{1,...,n_{sh}\}$ such that $i=\widehat{j}(K,n)$;} \\
    0, & \hbox{otherwise.}
  \end{array}
\right.
\end{split}
\end{align}
This implies that $\varphi_i(X_j)=\delta_{ij}$ for $1\leq i,j \leq N$. The set of functions $\{\varphi_1,...,\varphi_N\}$  are the global shape functions.
We define the discretized finite element space $H_h^1(\Omega)$ as 
$H_h^1(\Omega) = span\{\varphi_{j},j\in\{1,\ldots,N\}\}$.
We also consider $[H^1_h(\Omega)]^3$ and we define $H^1_{D,h}(\Omega,\R{3})$ as
$
[H^1_{D,h}(\Omega)]^3 = \{u\in [H^1(\Omega)]^3\, \vert \, u=0 ~~\text{on}~~ \overline{\partial\Omega_D}\cap\{X_1,\ldots,X_n\} \}
$. We now write the discretized elasticity problem as:
\begin{equation}
\label{eqa:WeakPDEDiscrete_new}
\left\{
  \begin{array}{ll}
    \text{Seek $u \in H_{D,h}^1(\Omega,\R{3})$ such that} &  \\\\
      B(u,v)=\int\limits_{\Omega}f\cdot v\, dx+\int\limits_{\partial\Omega_N}g\cdot v \,dA , \forall v\in H_{D,h}^1(\Omega,\R{3}) &
  \end{array}
\right.
\end{equation}  
By the usual theory of finite elements, see e.g. \cite{ErnGuerm04}, one can see that the discretized solutions of \eqref{eqa:WeakPDEDiscrete_new} converge in $[H^1(\Omega)]^3$ to the weak solution of the continuum problem \eqref{eqa:ElasticityStrong}, if $h\to 0$. 

\subsection{Discretization of the Probabilistic Model} % Sub-sub-section

For the numerical approximation of the failure probability we have to  to compute the cost functional
$J_\Omega=J(\Omega,u)$ which is an integral over the surface
$\partial\Omega$. We denote by $\mathcal{N}_h$ the collection of the boundary faces $F$ of finite elements $K=K(F)\in \mathcal{T}_h$ that lie in $\partial\Omega$.
The computation of surface integral $J_{sur}(\Omega,u)$ reduces to
evaluating integrals over each element in the collection
$\mathcal{N}_h$ as follows:
\begin{align}
\int\limits_{\partial\Omega}\left(\frac{1}{Ni_{det}(\sigma(x))}\right)^{\bar{m}}
dA=\sum_{F\in\mathcal{T}_h}\int\limits_{F}\left(\frac{1}{N_{i_{\rm det}}(\sigma(x))}\right)^{\bar{m}} dA.
\end{align}
Let $T_F: \widehat{F}\rightarrow F$ be a
$\mathcal{C}^1$-diffeomorphism mapping the geometric reference face
$\widehat{F} \subset \R{2}$ to any face $F$ in
$\mathcal{N}_h$. Let $J_F(\widehat{x})=\frac{\partial
T_F(\widehat{x})}{\partial \widehat{x}}\in \R{3,2}$ be the Jacobian matrix of
the mapping $T_F$ at $\widehat{x}$ and $F(\widehat{x})=(J_F(\widehat{x}))^TJ_F(\widehat{x})$ be the
Gram matrix. The change of variables
$x=T_K(\widehat{x})$ yields :
\begin{align}
\int\limits_{F}\left(\frac{1}{N_{i_{\rm det}}(\sigma(x))}\right)^{\bar{m}}
dA=\int\limits_{\widehat{F}}\left(\frac{1}{Ni_{det}(\sigma(T_F(\widehat{x})))}\right)^{\bar{m}} \sqrt{\det(g_F(\widehat{x}))}
d\widehat{A}
\end{align}
We consider a quadrature on $\widehat{F}$ defined by
$l_{q}^{F}$ Gau\ss $~$ points
$\{\widehat{\xi_{1}^{F}},...,\widehat{\xi_{l_{q}^{F}}^{F}}\}$
and $l_{q}^{F}$ weights
$\{\widehat{\omega_{1}^{F}},...,\widehat{\omega_{l_{q}^{F}}^{F}}\}$. We thus get
\begin{equation}\label{surface}
J(\Omega,u)\approx\sum_{F\in\mathcal{T}_h}\sum_{l=1}^{l_q}\omega_{lF}\left(\frac{1}{Ni_{det}
(\sigma_h(\xi_{lF}))}\right)^{-\bar{m}}.
\end{equation}
where we have set By setting
$\omega_{lF}=\widehat{\omega_l}\sqrt{\det(g_F(\widehat{\xi_{l}^{F}}))}$ and
$\xi_{lF}=T_F(\widehat{\xi_{l}^{F}})$.

For the computation of $J(\Omega,u)$ one has to evaluate the finite element stress $\sigma_h(x)=\lambda \nabla\cdot u_h(x)\mathcal{I}+2\mu \varepsilon(u_h(x))$ which can be computed on the basis of the first derivative $\nabla u_h$ of the finite element displacement field $u_h$.
By expanding $u$  in the global basis functions $\theta_j$, we get
$u(\xi)=\sum_{j=1}^{N}u_j\theta_j(\xi)=\sum_{K\in\mathcal{T}_h}\sum_{m=1}^{n_{\rm sh}}u_{j_{(K,m)}}\widehat{\theta}_m \circ T_{K}^{-1}(\xi)$ so that
\begin{equation}
\label{eqa:Jacobi}
\nabla u_h(\xi)=\sum_{m=1}^{n_{\rm sh}}u_{\widehat{j}(K,m)}\otimes( J_{K}(\widehat{\xi})^{T})^{-1}\widehat{\nabla} \widehat{\theta}_m (\widehat{\xi})),~\mbox{ for } \xi \in K\mbox{ and }\xi=T_K(\hat\xi).
\end{equation}
where $J_K(\widehat{\xi})=\widehat{\nabla} T_{K}(\hat\xi)=\sum_{j=1}^{n_{\rm sh}}\widehat\nabla\widehat\theta_j(\widehat{\xi})X_{K,j}$ be the Jacobian matrix of
the mapping $T_K$.
\subsection{The Example of a Radial Compressor}

As an example, we present the results the probabilistic life calculation of a radial compressor from the CalculiX FEA Toos suite\footnote{www.calculix.de} by G. Dhondt and K. Wittig, see also \cite{SSGBRK,Schmitz} for other numerical studies. The component consists out of 7 segments with 2 blades of different length, each. The FEA discretization contains 1 302 brick elements with 20 local degrees of freedom and reduced quadrature $l_q=8$. In total, the model contains 47 971 nodes. The surface quadrature is chosen as $l_q^F=16$ in order to account for the non linearity of the integrand in the objective functional.    

The material for the compressor is AlSi - C355, however, as CMB and RO parameters are not publicly available, we use those of the alloy AlMgSi6082, which is material with similar properties \cite{Alum}. We also note that the scaling procedure in \cite{SBKRSG} has been applied.

\begin{figure}
\centerline{\includegraphics[width=.35\textwidth]{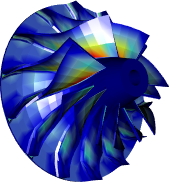}~~~~~~\includegraphics[width=.55\textwidth]{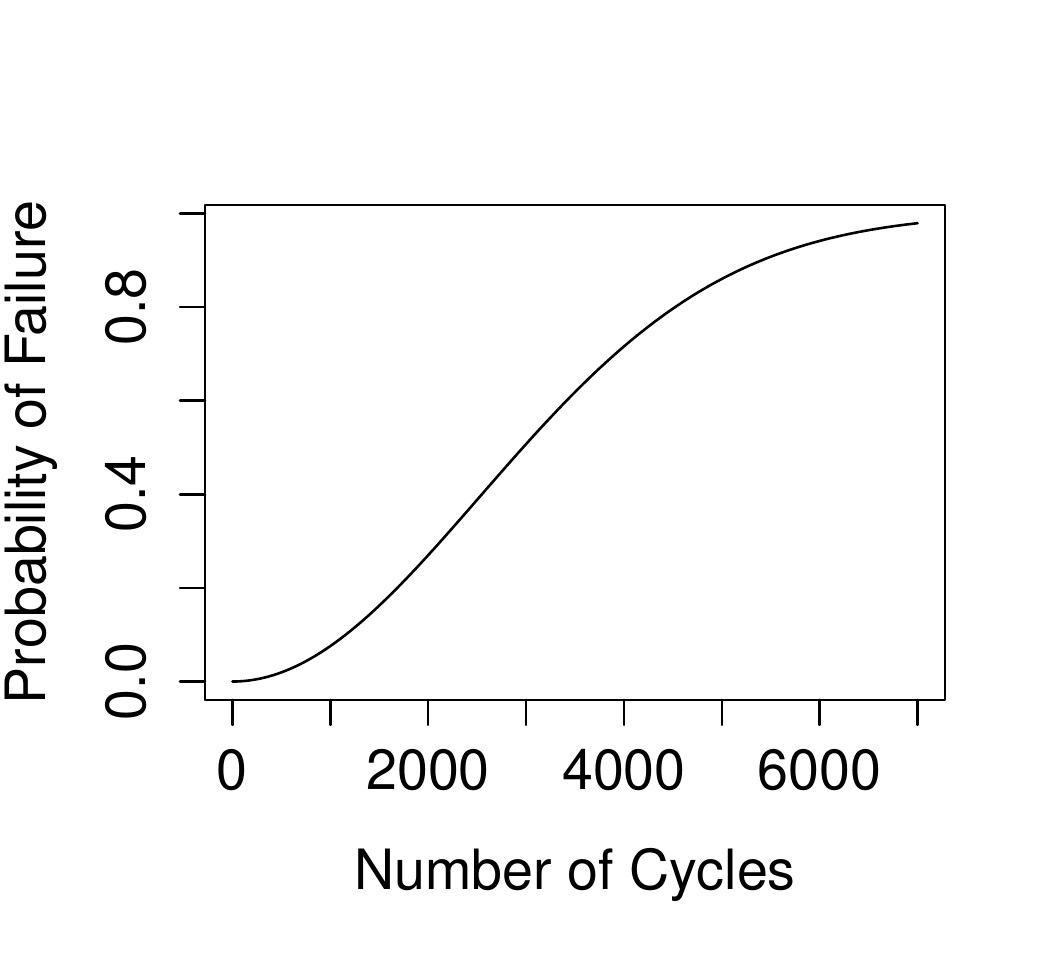}}

\caption{Local crack formation intensity for $J$ (left) and the failure probability over load cycles (right)}
\label{fig:Compressor}
\end{figure}

Gas pressure surface loads are neglected. The volume force is a centrifugal load which stems from a rotation speed of 110000 rpm. Figure \ref{fig:Compressor} shows the result of the calculation. The result nicely fits to the safe  life of app.\ 2000 cycles mentioned by Dhont and Wittig. 

\section{\label{sec:MultiScale}Microscopic Models for Failure Times}

In the following we give a short summary on investigations of scatter in LCF life caused by the presence of random grain orientation and resulting scatter in shear stresses of the slip planes of a face centred cubic to which a multi-axial load is applied to. By this work, we intend to derive physics based models for the hazard of mechanical failure in order to replace the purely empirical Weibull and Gompert's models in the future via a multi scale modelling approach.

\subsection{Crystal Properties and Probabilistic Schmid Factors}
We consider a crystal with face centered cubic lattice and planes of most dense packages being equal to the 4 slip planes forming a tetrahedron. As in each case two different edges of a slip plane build a slip system the fcc crystal has 12 of them. The unit cell of an fcc crystal with appropriate slip planes and systems is shown in Figure \ref{fig:slipSystem}.
\begin{figure}[t]
\begin{center}
\includegraphics[scale=0.6]{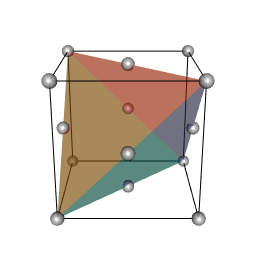}
\end{center}
\caption{Slip planes and systems in the face centred cubic lattice.}\label{fig:slipSystem}
\end{figure}
Let $n_i$ be the normal of slip plane $i$, $i=1,...,4$, and $s_{i,j}$, $j=1,2,3$, the vector of the $j$-th slip system related to slip plane $i$.\\
Considering the material to be isotropic, the random orientation of a grain\footnote{Structure of several unit cells having the same orientation.} can be represented by a random tranformation in form of a $3 \times 3$ rotation matrix $U \in$ $SO(3)$\footnote{$U$ is distributed according to the Haar measure.}, acting on $n_i$ and $s_{i,j}$. \\
Hence slip planes and systems become probabilistic via
\begin{equation}
n_i(U)=Un_i,     s_{i,j}(U)=Us_{i,j}\,.
\end{equation}
The shear stress in direction of a slip system $s_{i,j}$, $i=1,...,4$, $j=1,2,3$ and the according maximal shear stress given the stress tensor $\sigma$ can be calculated as follows
\begin{equation}
\tau_{i,j}=n_i\cdot \sigma \cdot s_{i,j}, \tau=\max_{i,j}|\tau_{i,j}|.
\end{equation}
Considering the random transformation acting on the slip planes and systems as shown in $(1)$ one can obtain probabilistic shear stresses as well as their maximal value depending on rotation $U$,
\begin{equation}
\tau_{i,j}(U)=Un_i\cdot \sigma \cdot Us_{i,j}, \tau=\max_{i,j}|\tau_{i,j}(U)|.
\end{equation}
As the probability distribution of $U$ is uniquely determined one can carry out simulations generating random rotation matrices or apply an analytical method to obtain values of the distribution and density functions of maximum shear stress to examine the differences depending on the respective loading state. Though usually the probabilistic Schmid factor is defined as $m(U)=\tau(U)/\sigma$, we can neglect this by normalisation and use $\tau(U)$ directly instead.\\ \\
Assuming an uniaxial stress state the stress tensor contains only one value not equal to zero,
\begin{equation}
\sigma=\sigma_k\cdot e_k\otimes e_k ,
\end{equation}
where $k=1,2,3$ and $e_k\otimes e_k=(e_ie_j)_{i,j=1,2,3}$  depending on the $k$-th stress direction.\\

\subsection{Multiaxial Stress States}

Let us consider three-dimensional loading conditions comprising stresses in normal directions and shear stresses in all directions of space. It is a frequent practice to calculate the so-called equivalent tensile stress\footnote{Or von Mises stress.} $\sigma_v$\footnote{According to the von Mises shape modification hypothesis hydrostatic stress conditions with similar principal stress in all directions lead to a value of zero.} from $\sigma$ as a comparable value for multiaxial stress states to an uniaxial tensile test,
\begin{equation}
\sigma_v=\sqrt{\frac{1}{2}\lbrack(\sigma_I-\sigma_{II})^2+(\sigma_{II}-\sigma_{III})^2+(\sigma_{III}-\sigma_{I})^2\rbrack},
\end{equation}
where $\sigma_I,\sigma_{II}, \sigma_{III}$ are the principal stress values.\\\\
The disantvantage of using the equivalent stress is that it does not take into account the relation between the individual principal stress values, hence the possible impacts of changes in loading conditions having equal von Mises stress on shear stresses acting on the slip systems are disregarded.\\
Assuming the principal stress tenses to be sorted according to their absolute size,
\begin{equation}
|\sigma_I|>|\sigma_{II}|>|\sigma_{III}|,
\end{equation}
we introduce the parameter $\kappa\in\lbrack 0,1\rbrack$ describing the relation of $\sigma_{III}$ and $\sigma_{II}$ to $\sigma_I$,
\begin{equation}
\kappa=\frac{|\sigma_{III}-\sigma_{II}|}{|\sigma_I|}.
\end{equation}
Parameter $\kappa$ is only equal to zero in the uniaxial loading state, the direction is the one having the largest absolute value of principal stress $\sigma_{I}$. As noticeable on Figure \ref{fig:Uniaxial} and \ref{fig:Multiaxial} the densities of the distributions of maximum shear stress differ considerably, either in mean or scatter.
\begin{figure}[htb]
    \centering
    \begin{minipage}[t]{0.5\linewidth}
        \centering
        \includegraphics[width=\linewidth]{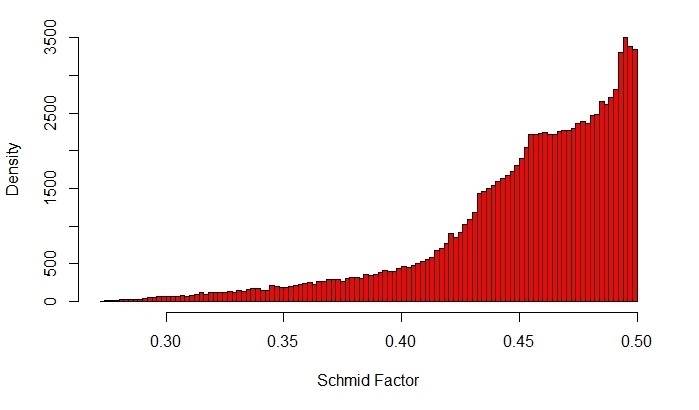}
        \caption{Uniaxial Load State:$\kappa=0$}\label{fig:Uniaxial}
    \end{minipage}% <- sonst wird hier ein Leerzeichen eingefügt
    \hfill
    \begin{minipage}[t]{0.5\linewidth}
        \centering
        \includegraphics[width=\linewidth]{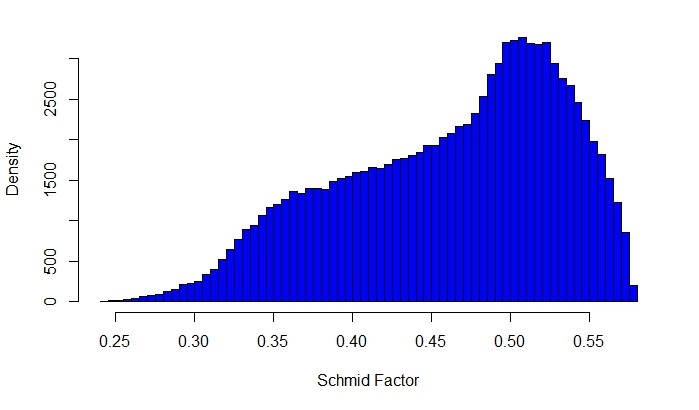}
        \caption{Multiaxial Load State: $\kappa=1$} \label{fig:Multiaxial}
    \end{minipage}
\end{figure}

\subsection{Impacts on Durability}
From the results for the distributions of maximum shear stress depending on the relation of principal stress values one can deduce the impact on the distributions of LCF life, respectively cycles to crack initiation $N_i$.
Since stress-life curves are not commonly used we adjust the strain amplitude by means of the Schmid factors. \\ 
Therefore consider the Ramberg-Osgood relation \eqref{eqa:Ramberg-Osgood} with $\varepsilon_a=RO(\sigma_a)$, 
%=\frac{\sigma_a}{E}+\biggl(\frac{\sigma_a}{K}\biggr)^{1/n'},
inverse $\sigma_a=RO^{-1}(\varepsilon_a)$ and the strain amplitude being adjusted by maximum shear stress $\tau(U)$,
\begin{equation}
\varepsilon_a(\tau(U))=RO\Biggl(\frac{\tau(U)}{\vartheta}\cdot RO^{-1}(\varepsilon_a)\Biggr),
\end{equation}
where $\vartheta$ is the expected value of Schmid factors under uniaxial loading conditions. 

Hence the inverted relation of the Schmid factor adjusted strain amplitude and the Coffin-Manson Basquin equation can be applied and we obtain the number of cycles to crack initiation depending on calculated maximal shear stress  by using
\begin{equation}
N_i(\tau(U))=CMB^{-1}(\varepsilon_a(\tau(U))).
\end{equation}
As one can notice deviating results for the distributions of $N_i$ depending on the principal stress combination, this implies that $\kappa$ has a large impact on the LCF life and the differences we can see in Figure \ref{fig:Uniaxial} and Figure \ref{fig:Multiaxial} can be transferred.
\begin{figure}[htb]
    \centering
    \begin{minipage}[t]{0.5\linewidth}
        \centering
        \includegraphics[width=\linewidth]{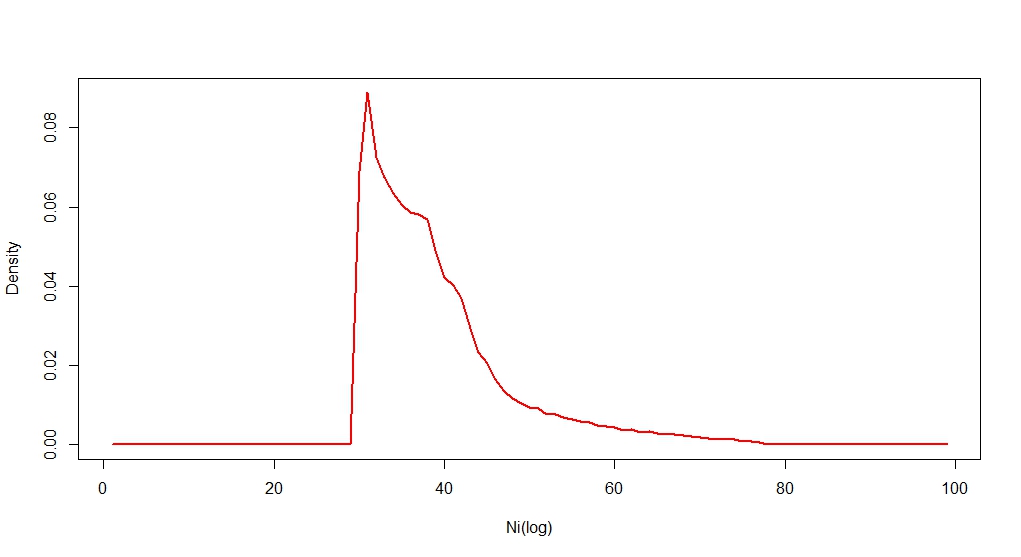}
        \caption{Uniaxial Load State:$\kappa=0$}\label{fig:LifeUni}
    \end{minipage}% <- sonst wird hier ein Leerzeichen eingefügt
    \hfill
    \begin{minipage}[t]{0.5\linewidth}
        \centering
        \includegraphics[width=\linewidth]{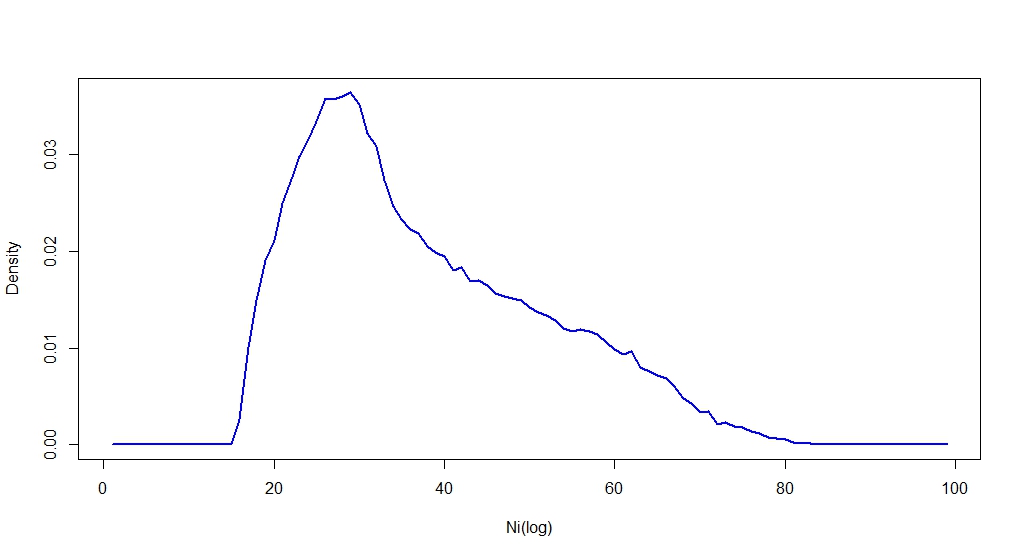}
        \caption{Multiaxial Load State: $\kappa=1$}\label{fig:LifeMulti}
    \end{minipage}
\end{figure}

\subsection{An Outlook on Multiscale Modeling Approaches}

A complex component, like a gas turbine blade, has a surface made out of thousands of grain facets. However, highly loaded spots may only contain a few of such facets. In order to accurately predict the LCF life of complex components based on the microscopic investigations exposed here, let us consider a surface $A$ which is exposed to a constant stress $\sigma$. Let $\kappa=\kappa(\sigma)$ be the associated parameter measuring the multiaxiality from $\kappa=0$ (uniaxial) to $\kappa=1$ (maximally multi-axial). Furthermore, let $N_g$ be the number of gains on the given surface and let $S_{j,N_i,\sigma}(t)$ be the survival function of grain $j$ obtained from Figures \ref{fig:LifeUni} and \ref{fig:LifeMulti} and the related distributions for an intermediate $\kappa$. We note that, apart from the parameter $\kappa(\sigma)$, $S_{N_i}(t|\sigma)=S_{j,N_i}(t|\sigma)$ also depends on the von Mises stress $\sigma^{\rm el}$. Following the general logic of extreme value theory \cite{EscobarMeeker}, the total probability of survival is thus 
\begin{equation}
S_A(t)=\prod_{j=1}^{N_g}S_{j,N_i,\sigma}(t)=(1-(1-S_{N_i}(t|\sigma)))^{N_g}.
\end{equation}
Let us consider a $t$ small enough such that $1-S_{N_i}(t|\sigma)\approx -H_{N_i}(t|\sigma)$ with $-H_{N_i}(t|\sigma)$ the cumulative Hazard rate of a single grain. Furthermore, the number of grains $N_g$ is roughly $N_g=\frac{1}{\mu_g}A$, where $\mu_g$ is the average grain surface. We then obtain in the given approximation
\begin{equation}
S_A(t)\approx \left(1-\frac{\frac{|A|}{\mu_g}H_{N_i,\sigma}(t)}{N_g}\right)^{N_g}\approx e^{-\frac{A}{\mu_g}H_{N_i}(t|\sigma)},  
\end{equation}
provided $\frac{A}{\mu_g}H_{N_i}(t|\sigma)\ll 1$, which says that the entire PoF of $A$ is small. Here $|A|$ stands for the surface volume of $A$.

It is now straight forward to propose the following physics based local model for LCF, namely
\begin{equation}
\label{eqa:newModel}
S_\Omega(t)=\exp\left\{-\frac{1}{\mu_g}\int_{\partial\Omega}H_{N_i}(t|\sigma(x))\, dA\right\},
\end{equation}
which is based on the single grain cumulative hazard rates $H_{N_i,\sigma(x)}(t)$ obtained in this section. We thus see that this model can also be formulated in the framework  of crack initiation processes with
\begin{align}
\varrho_\Omega(t,x)=\frac{1}{\mu_g}h_{N_i}(t|\sigma(x)),
\end{align}
where $h_{N_i}(t|\sigma(x))=H_{N_i}'(t|\sigma(x))$ is the hazard rate of the single grain distribution. 

We note that this model is only to illustrate the principle of the multi-scale modelling, but does not necessarily give the correct grain size dependence, as small grain structures usually lead to a longer component life. It therefore seems to be necessary, to combine the derivation of  \eqref{eqa:newModel} with microscopic crack percolation up to a fixed critical crack size as e.g. described  in \cite{Sornette} to get the correct grain size dependency. We will come back to this in future work.

\section{\label{sec:OpimalReliability}Shape Optimization -- Existence of Optimal Shapes}

\subsection{Optimal Reliability as a Problem of Shape Optimization}

We have seen that in the case of the local Weibull model and the local Gompert's model, the question of optimal reliability in all three variants of Def.\ \ref{def:OptRel} can be understood as a minimization of $J_\Omega$ in the shape $\Omega$, where $J_\Omega=J(\Omega,u)$ is the proportional hazard constant. Here we stressed the dependency of the solution $u$ of \eqref{eqa:ElasticityStrong}, which is though the calculation of elastic-plastic strain amplitude field $\varepsilon_a^{\rm el-pl}(x)$. We note that this field is a local function of $\nabla u$. Consequently, we can define an objective functional
\begin{equation}
\label{eqa:ObjectiveFunctProbLCF}
J(\Omega,u)=\int_{\partial\Omega}\left(\frac{1}{N_{i_{\rm det}}(\nabla u)}\right)^ m dA,
\end{equation}
or a related quantity for the local Gompert's model, see \eqref{eqa:locG}. All three versions of the optimal reliability problem \ref{def:OptRel} are then equivalent to the following problem of shape optimization:
\begin{align}
\min ~& J (\Omega,u(\Omega)) \nonumber \\
\text{s.t.} ~&u(\Omega) \text{ satisfies a given condition } P(\Omega), \label{P} \tag{$\mathbb{P}$}  \\
&\Omega \in \O,  \nonumber 
\end{align}
where $ P(\Omega) $ can be a PDE, ODE or variational inequality for example. In our case $ P(\Omega) $ is given by the PDE-formulation of linear isotropic elasticity, see \eqref{eqa:ElasticityStrong}
on $ \Omega \subset \R{3} $.

\subsection{\label{sec:Basics}Basic Notations and Abstract Setting for Shape Design Problems}\label{Sec: basic notations}
We will now summarize abstract methods presented for example in \cite{BucBut05,ShapeOpt,SokZol92} to show existence of solutions to shape optimization problems.

%Here, $ \RO_{N} \cupdot \RO_{D} $ is a partition of the domain's boundary, where on $ \RO_{N} $ a force surface density $ g\restriction_{\RO_{N}} $ is imposed and $ \RO_{D} $ is clamped. Let $\nu $ be the outward normal on $ \RO $. The load vector field $ f:\Omega \to \R{3} $ corresponds to a force imposed on the volume of $ \Omega $ and every solution $ u:\overline{\Omega}\to\R{3} $ is called displacement field on $ \Omega $. $ I $ denotes the identity on $ \R{3} $. In addition we assume that the Lamé-coefficients $\lambda,\mu>0$ are constants.

A solution of \ref{P} is sought as a set $\Omega$ in the familiy of admissible domains $\O$, containing possible candidates of shapes. This set is contained in a larger system $\Os$ on which some kind of convergence $\On \xrightarrow[]{\Os} \Omega $ as $n \to \infty$ depending on the respective problem is given. Further on, all possible solution $ u(\Omega) $ of $ P(\Omega) $ shall be contained in a state space $V(\Omega)$ of real functions on $\Omega$ for every $\Omega \in \Os$.

Since solutions to $ P(\Omega),\, \Omega \in \O $ are defined on changing sets there should be a suitable definition of convergence denoted by $y_{n} \rightsquigarrow y$ as $n \to \infty$ where $y_{n} \in V(\On)$, $\On \in \Os$. Moreover, require that any subsequence of a convergent sequence tends to the same limit as the original one. \\
We further assume that every state problem $P(\Omega)$ ha a unique soultion for eyery $ \Omega \in \O $ what makes us able define the map $u:\Omega \to u(\Omega) \in V(\Omega)$. The resulting set $\G=\lbrace(\Omega,u(\Omega))\, \vert \, \Omega \in \O\rbrace$  is called the graph of $u$ restricted to $\O$. The set $\mathcal{G}$ is called compact iff every sequence $\seq{\On,u(\On)} \subset \mathcal{G}$ has a subsequence $\subseq{\Omega_{n_{k}},u(\Omega_{n_{k}})}$ where
\begin{align}
	\begin{array}{r c l}
	 \Omega_{n_{k}} &\xrightarrow[]{\Os} & \Omega \\
	  u(\Omega_{n_{k}}) &\rightsquigarrow & u(\Omega)
	\end{array}
	\label{Def: Gcomp}
\end{align}

\noindent as $ k \to \infty $ for some $(\Omega, u(\Omega)) \in  \mathcal{G}$.

\vspace{1em}

\noindent A cost functional $\mathcal{J}$ on $\Os$ maps a pair $(\Omega,y)$,  $\Omega \in \Os$, $y \in V(\Omega)$ onto $ J (\Omega,y)$. for such functionals, lower semi-continuity is defined as follows: 

\noindent Let the sequences $\seq{\On}$ in $\Os$ and $\seq{y_{n}},\, y_{n} \in V(\On)$ be convergent against $\Omega \in \O$ and $y \in V(\Omega)$, respectively. Then
\begin{align}
	\left.
	\begin{array}{r c l}
		\On & \xrightarrow[n \to \infty]{\Os} & \Omega \\
		 y_{n} &  \underset{ n \to \infty}{\leadsto}& y
	\end{array} \right\rbrace \Rightarrow \liminf_{n \to \infty} J (\On,y_{n}) \geq J (\O,y)
\end{align}

\noindent Now, let $\O$ be a subfamily of $\Os$ and let $u(\Omega)$ be the unique solution of a given state problem $P(\Omega)$ for every $\Omega \in \O$. An \textit{optimal shape design problem} can be defined by  
\begin{equation} \label{opt_shape_des}
	\text{Find } \Omega^{*} \in \O\text{ such that \ref{P} is solved}\,.
\end{equation}

\noindent The following theorem based on the general fact that lower semicontinuous functions always possess a minimum on a compact set provides conditions for the existence of optimal shapes.

\begin{thm}\label{Thm: existence of optimal shapes} \cite[Ch. 2]{ShapeOpt}
Let $\Os$ be a family of admissible shapes with a subfamily $\O$. It is assumed that every $\Omega \in \O$ has an associated state problem $P(\Omega)$ with state space $V(\Omega)$ which is uniquely solved by $u(\Omega) \in  V(\Omega)$. Finally, require
\begin{itemize}
\item[(i)] compactness of $\mathcal{G}$,
\item[(ii)] lower semi-continuity of $J$.
\end{itemize}  
Then there is at least one solution of the optimal shape design problem.
\end{thm}

  \subsection{\label{sec:Domains}$ C^{k,\phi} $-Admissible Domains via Defomation Maps} \label{Sec: form_design}

Now, we have adjust these terms to our present problem.
%Later on we will impose coupled boundary value problems of linear elasticity and heat-equation on these domains. 

We choose $\Omega_{0} \subset \R{3}$ to be a $C^{k,\phi}$-domain for some  $\phi \in (0,1]$. Further let $B:=\Umg{r}{z} \subset \Omega_{0}, \, z \in \text{int}(\hat{\Omega})$ be a ball in its interior with distance $D:=\dist(\Umg{r}{z}, \partial \hat{\Omega})>0$ from the boundary. Then we set $\Omega_{b}:=\Omega_{0} \setminus B $ and choose a fixed constant $\mathcal{K}>0$.
 
The elements of 
\begin{equation}
\Uad_{k,\phi}:=\left\lbrace \Phi \in \left[\D{k,\phi}{\overline{\Oext}}\right]^3 \, \Big\vert \, \Norm{\Phi}{\Ck{k,\phi}{\Oext}{3}}\leq \mathcal{K},\, \Norm{\Phi^{-1}}{\Ck{k,\phi}{\Oext}{3}}\leq \mathcal{K} \right\rbrace,  
\end{equation}
where $\left[\D{k,\phi}{\Omega,\Omega'}\right]^n$ is the set of $C^{k,\phi}(\Omega,\Omega')$-diffeomorphisms, are called design variables. The set of admissible shapes assigned to $ \Omega_{b} $ is then given by
\begin{equation}
 \O_{k,\phi}:= \left\lbrace \Phi(\Omega_{b}) \, \vert \, \Phi \in \Uad_{k,\phi} \right\rbrace.
 \end{equation} 
Under these conditions $ \Uad_{k,\phi} $ is compact in the Banach space $ \left(\Ck{k,\phi'}{\overline{\Oext}}{3},\Norm{.}{\Ck{k,\phi'}{\Oext}{3}} \right) $ for any  $ 0\leq \phi' <\phi $ and $ k \in \N{}$ (compare \cite[Lemma 6.36]{GilbTrud}), what makes it obvious to define convergence of sets through $ \left[C^{k,\phi'}\right]^3 $-convergence of admissible functions. That means
$ \Omega_{n} \overset{\O}{\longrightarrow} \Omega, \, n \to \infty $ if and only if the  corresponding sequence $ \seq{\Phi_{n}} \subset \Ck{k,\phi}{\overline{\Oext}}{3}\,  n\in \N{}$   tends to $ \Phi \in \Ck{k,\phi'}{\overline{\Oext}}{3} $, where $\Phi_{n}(\Omega_{b})=\Omega_{n}\Phi(\Omega_{b})=\Omega $.

\subsection{Schauder estimates for Linear Elasticity Equation}\label{Subsec: solutions for elasticity equation}

We now invoke on the regularity results for the linear elasticity problem presented in Theorems 6.3-5 and 6.3-6 in \cite{Cia1988}, which are needed to proof Theorem \ref{Thm: schauder estimates u} below.
Accordingly, we set $ \Uad:=\Uad_{4,\phi} $, $ \phi \in (0,1) $ for the set of feasible design-variables and  $ \O:=\O_{4,\phi} $ for the set of  admissible shapes.

%\begin{defn}[\bf{ State Problem for Linear Elasticity Equation}]\label{Def: state space for elasticity} $  $\\
%Let $\Omega:=\Phi(\Omega_{b}) $ be a $ C^{4,\phi} $-admissible shape for some $ \Phi \in \Uad $. The associated state problem $ \mathcal{P}(\Omega) $ is defined to be given by the elasticity equation \eqref{PO} where $ \partial\Omega_{D}=\partial B $ is the interior boundary and $ \partial\Omega_{N}=\partial\Omega \setminus \partial \Omega_{N} $ the complete exterior boundary\footnote{Note that the chosen decomposition of the boundary depends continuously on the choice of $ \Phi \in \Uad $. Moreover the two dimensional Lebesgue-measures of $ \partial\Omega_{N} $ and $ \partial\Omega_{D} $ are always bounded away from zero.}.
%\end{defn}

%\textcolor{red}{$C^{k,\phi'}$ if $k+\phi'\leq 4+\phi$, $ k \in \N{}, \phi'\in (0,1) $ }

Volume force densities can be easily defined as gravitational or centrifugal loads, such that $f_\Omega=f^{\rm ext}\restriction_\Omega$ for some $f\in \left[C^{1,\phi}(\Omega^{\rm ext})\right]^3$. The surface load $g$ generally depends on the shape $\Omega$ in a non trivial way. One idea to deal with those loads as sections with uniform bound on the fibre norm: 
\begin{equation}
 G^{ad}_{\phi}(\mathcal{O}):=\left\lbrace g:\O \to \dot\bigcup_{\Omega\in\mathcal{O}} \Ck{2,\phi}{\RO}{3}\, \left\vert \,\Norm{ g(\Omega)}{\Ck{2,\phi}{\RO}{3}}\leq k_{1}<\infty ~\forall\Omega\in{\cal O}\right. \right\rbrace.
 \end{equation}   

One example for this construction is the restriction of a function $ g^{ext}\in 
\Ck{2,\phi}{\Oext}{3} $ that is defined on the larger set $ \Oext $ to $ \Omega $:
\begin{equation}
 g(\Omega):=g^{ext}\restriction_{\RO} \text{ with } k_1=\Norm{g^{ext}}{\Ck{2,\phi}{\RO}{3}}.
 \end{equation}

Due to it's construction, every shape $ \Omega \in \O $ has a Lipschitz-boundary and the associated Lipschitz constant can be chosen uniformly, what is proofed to be equivalent to a uniform cone property in \cite{Chen75}.  Moreover, each $ \Omega \in \O $ satisfies a hemisphere property where the corresponding hemisphere transformations of class $ C^{4,\phi'} $, $ \phi' \in [0,\phi] $ are uniformly bounded and so the boundary of every $ \Omega $ can be straigtened by Diffeomorphisms that can be estimated by the same constant. These facts, Korn's second inequality, the compactness of $ \partial \Omega,\, \Omega \in \O $ and Theorems 6.3-5 and 6.3-6 in \cite{Cia1988} lead to the following statement:

%\begin{lem} \cite[Lemma 5.5]{Hanno} \label{Lemma: uniform cone lemma} 
%	Let $ \mathcal{M} $ be a set of bounded domains in $ \R{n}$ with a uniform cone property and let $ \Omega \in \mathcal{M} $. Then, for every $ \varepsilon>0 $ there is a constant $ C(\varepsilon)>0 $ uniform with respect to $ \mathcal{M} $, such that  $ \Norm{v}{\C{0}{\Omega}} \leq \varepsilon \Norm{v}{\C{1}{\Omega}} + C(\varepsilon)\int_{\Omega}\vert v \vert dx  $ holds for all $ v \in \C{1}{\Omega} $.
%\end{lem}

\begin{thm} \label{Thm: schauder estimates u} \cite[Theorem 5.6, 5.7]{Hanno}\\ 
Recall the PDE \eqref{eqa:ElasticityStrong}, where $ \Omega=\Phi(\Omega_{b}) $ for some $ \Phi \in \Uad $.
\begin{itemize}
\item[(i)] Let $ f\in \Ck{1,\phi}{\overline{\Oext}}{3} $, $ g\in \Ck{2,\phi}{\overline{\Omega}}{3}$  for some $ \phi\in (0,1) $. Then there exists exactly one solution $ u\in \Ck{3,\phi}{\overline{\Omega}}{3} $ that satisfies 
\begin{equation}\label{eq: schauder inequality u}
\Norm{u}{\Ck{3,\varphi}{\Omega}{3}} \leq C \bigl(  \Norm{f}{\Ck{1,\phi}{\Omega}{3}} + \Norm{g}{\Ck{2,\phi}{\RO}{3}} +  \Norm{u}{\Ck{0}{\Omega}{3}}  \bigr). 
\end{equation}

\noindent for any $ \varphi \in (0,\phi) $ and some positive constant $ C $ independent from $ \Omega \in \O $.

\item[(ii)] Let $ f\in \Ck{1,\phi}{\overline{\Oext}}{3} $. Moreover, let $ g=g(\Omega) $ be the associated mapping to some $ g \in G^{ad} $. Then,  
\begin{equation}
\Norm{u}{\Ck{3,\varphi}{\Omega}{3}} \leq C
\end{equation} 
holds for any $ \varphi \in (0,\phi) $ and a constant $ C$ which can be chosen uniformly w.r.t. $ \O $.
\end{itemize}
\end{thm} 
Hence, for fixed $ \phi \in (0,1) $ we define the state space 
\begin{equation}
 V^{\varphi}(\Omega):= \Ck{3,\varphi}{\overline{\Omega}}{3}, \text{ for some }\varphi \in (0,\phi).\end{equation}

\subsection{\label{sec:OpShapes}Existence for the SO  and the Optimal Reliabiliy Problem}\label{Sec: exsistence of optimal shapes}

Now motivate existence of optimal solutions to shape optimization problems where the constraints are given by thermal elasticity and where the cost functionals are of very general class. 

The objective is to find an optimal shape $ \Omega=\Phi(\Omega_{b}) $ within the set of $ C^{4,\phi} $-admissible shapes $ \O  $ which minimizes a local cost functional $ J (\Omega,u) =J _{vol}(\Omega,u) +J _{sur}(\Omega,u) $, where
\begin{equation}\label{Def: local cost functional}
\begin{split}
	J _{vol}(\Omega,u)&=\int_{\Omega}\mathcal{F}_{vol}(x,u,\nabla u,\nabla^{2}u,\nabla^{3}u)\,dx\\[1ex]
	J _{sur}(\Omega,u)&=\int_{\RO}\mathcal{F}_{sur}(x,u,\nabla u,\nabla^{2}u,\nabla^{3}u)\,dA.
\end{split}
\end{equation}
Here $ u $ solves the state Problem \eqref{eqa:ElasticityStrong}. 
Owing to the trace theorem\footnote{Confer for example \cite[5.5]{Evans}} the appearing surface integrals lead to a loss of regularity. Therefore, these functionals are to singular to be treated with weak solution theory \cite{Hanno,Schmitz}. Hence, regularity theory \cite{Agm59,Agm64} and strong solutions are indispensable.  

Now we need a definition of convergence $ u_{n} \rightsquigarrow u  $ corresponding to compactness of the graph $ \mathcal{G}=\{(\Omega,T,u) \, \vert \, \Omega\in \O \} $ in terms of \eqref{Def: Gcomp} that reflects the regularity characteristics of the present elasticity problem. This can be achieved by extension: 

Let $ m,q \in \mathbb{N}$, $\beta\in (0,1)$. Then the operator $p^ {m,q,\beta}:\big[\C{q,\beta}{\overline{\Omega}}\big]^{m} \to \big[C^{q,\beta}_{0}(\Oext) \big]^{m} $ maps $ v \in \big[\C{q,\beta}{\overline{\Omega}}\big]^{m} $ to $ v^{ext} \in  \big[\C{q,\beta}{\overline{\Omega^{ext}}}\big]^{3} $, compare \cite[Lemma 6.37]{GilbTrud}.

Hence, $ u_{n} \rightsquigarrow u \text{ as }  n\to \infty$: $\Leftrightarrow  u_{n}^{ext}\to u^{ext} \text{ in  } \big[C^{3,\phi}_{0}(\Oext) \big]^{m}$, turns out to be an appropriate choice $ (m=3,q=0,\beta=\phi) $. Moreover, it holds that also the extended solutions can be estimated uniformly due to the following inequality:
\begin{equation}
\Norm{u_{n}^{ext}}{\C{3,\phi}{\Oext}}\leq C \Norm{u_{n}}{\C{3,\phi}{\Omega}},
\end{equation}
where $ C=C_{q} $ is independent of $ \Omega $ and $ \Oext $.

 \begin{lem}[\textit{\textbf{Compactness of the Graph}}]\label{Lemma: compactness of the graph} \cite{Hanno}\\
Let $ \seq{\Omega_{n}}=\seq{\Phi_{n}(\Omega_{b})} \subset \O $ be an arbitrary sequence, where on any $ \Omega_{n} $ the setting of Theorem \ref{Thm: schauder estimates u} is given. Let $ \seq{u_{n}} \subset \big[\C{3,\phi}{\overline{\Omega_{n}}}\big]^{3} $ be the sequence of solutions to the state problems $\mathcal{P}(\Omega_{n}) $. Then, $ u_{n}\in V^{\varphi}(\On) $  and the sequence $ \seq{\Omega_{n},u_{n}} \subset \mathcal{G} $ has a subsequence $\subseq{\Omega_{n_{k}},u_{n_{k}}}$ such that $ \Omega_{n} \xrightarrow{\O} \Omega $, $ \Omega=\Phi(\Omega_{b}) $ as $ k\to \infty $, as well as $ u_{n_{k}}\underset{k\to \infty}{\rightsquigarrow} u $ for the corresponding solution $ u \in \big[\C{3,\phi}{\overline{\Omega}}\big]^{3} $ to $ \mathcal{P}(\Omega) $, where $ u\in V^{\varphi}(\Omega) $, $ 0<\varphi<\phi $.
\end{lem}

Due to the uniform bounds this follows from arguments similar to the Arzéla-Ascoli Theorem, see \cite[Lemma 6.38]{GilbTrud}, and the uniform convergence of the solutions sequence $ \seq{u_{n}} $ in the $ \left[C^{3,\varphi}\right]^{3} $-Norm. 

This uniform boundedness and the compactness of the boundarys $ \RO $ combined with Lebesgue's Theorem lead to the next statement:

\begin{lem}[\textbf{\textit{Continuity of Local Cost Funktionals}}]
\label{Lemma:Continuity} \cite{Hanno}\\
Let $ \mathcal{F}_{vol},\, \mathcal{F}_{sur} \in C^{0}(\R{d})$ with $ d=3+\sum_{j=0}^{3}3^{j+1} $, and let the set $ \O $ consist of $ C^{0} $-admissible shapes. For $ \Omega \in \O $ und  $ u\in \C{3}{\overline{\Omega}},\, (m=3,q=3,\beta=0)$ consider the volume integral $  J_{vol}(\Omega,u) $ and the surface integral $ J_{sur}(\Omega,u) $.
Let $ \seq{\Omega_{n}}\in\O $ with $ \Omega_{n} \xrightarrow{\O}\Omega $ as $ n \to \infty $, $ \seq{u_{n}}\in \Ck{3}{\overline{\Omega_{n}}}{3} $ be a sequence with $ u_{n}\rightsquigarrow u $. Then,
\begin{itemize}
\item[(i)] $J _{vol}(\Omega_{n},u_{n})\to J_{vol}(\Omega,u)$ as $ n\to \infty $\,.
\item[(ii)] If the family $ \O $ consists only of $C^{1} $-admissible shapes, then $J _{sur}(\Omega_{n},u_{n})\to J_{sur}(\Omega,u)$ as $ n\to \infty $\,.
\end{itemize}
\end{lem}

\noindent Now all requirements of Theorem \ref{Thm: existence of optimal shapes}  are statisfied:

\begin{thm}[\textbf{\textit{Solution to the SO and Optimal Reliability Problem}}]~\\
\label{thm:SolSO}
Let the set of admissible shapes be $ \Uad_{k,\phi}$ with $k=4$, $ \phi \in (0,1) $. Then the shape optimization problem \eqref{opt_shape_des} with the objective functional \eqref{Def: local cost functional} and the mechanical elasticity state equation $ P(\Omega) $ given by \eqref{eqa:ElasticityStrong} has at least one solution $\Omega^*\in\mathcal{O}$. 

In particular, applying the above to \eqref{eqa:ObjectiveFunctProbLCF}, there exists at least one solution to the Optimal Reliability Problem in all versions (i)-(iii) in Def.\ \ref{def:OptRel}. The same holds for the local Gompert's model.
\end{thm}

\section{\label{sec:ShapeDerivatives}Continuous Shape Derivatives for Failure Probabilities}

\subsection{Basic Ideas of Shape Calculus}
In shape optimization, sensitivity analysis is an important
application which can provide necessary conditions for the
existence of optimal shapes and which can establish a link to
efficient numerical optimization schemes by means of the Eulerian
derivative.

In this section, we first introduce into basic terms of shape
sensitivity analysis and focus on \cite{SokZol92} and
\cite{ShapeOpt} what has also been summarized
in \cite{Schmitz}. A sound mathematical variation of domains
is crucial for differentiation of functions which can be defined
on different domains. For this purpose, the generation of a family
of perturbations $\{\Omega_t\}_{0\leq t<\varepsilon}$ of a given
domain $\Omega\subset\mathbb{R}^3$ has to be given where every
domain is simply connected and has a $C^k$-boundary, $k\geq1$. The
family $\{\Omega_t\}_{0\leq t<\varepsilon}$ is defined by a family
of injective transformations $\{\Phi_t\}_{0\leq
t<\varepsilon}$ with $\Phi_t(\Omega)=\Omega_t$ for $0\leq
t<\varepsilon$. As usually small deformations of $\Omega$ are
considered.  Let $\Omega^ {\rm ext}\subset\mathbb{R}^3$ be a domain with a
piecewise $C^k$-boundary for $k\geq1$ and with an almost
everywhere existing outward unit normal field $\nu$. Defining the
set of admissible speed vector fields by
\begin{align}
\begin{split}
V^k(\Omega^ {\rm ext})&=\left\{V\in
C^k_0(\mathbb{R}^3,\mathbb{R}^3)\,|\,V\cdot\nu=0\quad\textrm{on
$\partial \Omega^ {\rm ext}$ except for singular
points,}\right.\\
&\left.\hspace{36mm} V(x)=0\quad\textrm{for all singular
points}\right\},
\end{split}
\end{align}
and equipping $V^k(\Omega^ {\rm ext})$ with the topology induced by
$C^k_0(\mathbb{R}^3,\mathbb{R}^3)$. Let $V\in
C\left([0,\varepsilon),V^k(D)\right)$ be a vector field and let $\Phi_t=\Phi_t(V)$ for each $t\in\R{}$ be the associated flow.  

%If $\Omega\subset\mathbb{R}^n$ is the initial geometry one can
%choose an evolutionary process to deform $\Omega$. Let
%$\{F_t\}_{t\in[0,t_0]}$ be a one-parameter family of mappings with
%$F_t:\Omega\rightarrow\mathbb{R}^n$ for $t\in[0,t_0]$ and
%$F_0=\textrm{id}$. The new geometry of $\Omega$ at time
%$t\in[0,t_0]$ is defined by $\Omega_t=F_t(\Omega)$. Moreover, the
%mappings $F_t$ are required to be injective such that
%$F_t(\mathring{\Omega})=\mathring{\Omega}_t$ and
%$F_t(\partial\Omega)=\partial\Omega_t$. There are two preferred
%classes of mappings $F_t$. On the one hand there is the
%perturbation of identity method which is given by a so-called
%velocity field
%\begin{equation}
%V\in
%W^{k,\infty}\left(\mathbb{R}^N,\mathbb{R}^N\right)\quad\textrm{or}\quad V\in C^k\left(\mathbb{R}^N,\mathbb{R}^N\right)%V\in(H^{1,\infty}(\Omega))^n
%\end{equation}
%and the approach $F_t=\textrm{id}+tV$ for $t>0$. On the other hand
%there is the speed method described by a sufficiently smooth
%vector field $V:\mathbb{R}\times\Omega\rightarrow\mathbb{R}^n$ and
%the flow equation $\frac{dx}{dt}=V(t,x),x(0)=x_0$. Then, $F_t(x)$
%is defined as the value of the solution of the flow equation.
Now, we can define the term of shape differentiability according
to \cite{SokZol92}:
\begin{defn}\label{definition_shape_derivative_2}\textbf{(Shape Differentiability, Eulerian Derivative)}\\
Let $D\subset\mathbb{R}^n$ be open and $\Omega\subset D$ be
measurable. Furthermore, let $J:(\Omega,y)\longmapsto
J(\Omega,y)\in \mathbb{R}$ be a cost functional and let
$u(\Omega)$ be the unique solution of a state problem given in
$\Omega$. If the so-called Eulerian derivative
\begin{equation}
    dJ(\Omega)[V]=\lim_{t\rightarrow0^+}\frac{J(\Omega_t,u(\Omega_t))-J(\Omega,u(\Omega))}{t}
\end{equation}
exists for all vector fields $V\in V^k(\Omega^ {\rm ext})
$ and $V\mapsto dJ(\Omega)[V]$ is
linear and continuous, $J$ is called shape differentiable at
$\Omega$.
\end{defn}

Further important definitions are the material and the shape
derivative, see \cite{SokZol92}. We here present a version that is adapted to the use of elliptic regularity in the previous section: 

\begin{defn}\label{definition_material_derivative}\textbf{(Material Derivative, Shape Derivative)}\\
Let $\Omega^{\rm ext}\subset\mathbb{R}^3$ be open with a piecewise smooth
boundary and let $\Omega\subset \Omega^{\rm ext}$ be a bounded domain with a
$C^k$-boundary. Then, the material derivative $\dot{y}(\Omega;V)$
of $y(\Omega)\in \left[C^{k,\varphi}(\Omega)\right]^3$ with $s\in[0,k],1\leq p<\infty$
in the direction of a vector field $V\in V^ k(\Omega^ {\rm ext})$ exists if
\begin{equation}
\dot{y}(\Omega;V)=\lim_{t\rightarrow0^+}\frac{y(\Omega_t)\circ
\Phi_t(V)-y(\Omega)}{t}
\end{equation}
is an element of $\left[C^{3,\varphi}(\Omega)\right]$. Convergence in the above equation is in the $C^{3,\varphi}$ topology. The shape derivative $y'(\Omega;V)$ of
$y(\Omega)$ in the direction $V$ is defined as
\begin{equation}
y'(\Omega;V)=\dot{y}(\Omega;V)-\nabla y(\Omega)\cdot V(0),
\end{equation}
provided that $\nabla y(\Omega)\cdot V(0)\in \left[C^{k-1,\varphi}(\Omega)\right]^3$ for
all $V$. %$V\in C_0^k(\mathbb{R}^3,\mathbb{R}^3)$.
\end{defn}
The advantage of the shape derivative is that $y'(\Omega,V)=0$ holds if $V\restriction_{\partial\Omega}$ is tangential to $\partial\Omega$, which is not true for the material derivative.
 
 Under further
assumptions, one can show that for a shape differentiable cost
functional $J$ there exists a scalar distribution
$\psi\in\left(C^k_0(\partial\Omega)\right)'$ such that
\begin{equation}\label{SO.1.1.2}
    dJ(\Omega)[V]=\int_{\partial\Omega}(V(x)\cdot\nu(x))\,\psi(x)\,dA
\end{equation}
with $V$ a sufficiently smooth vector field and $\nu$ normal of
$\partial\Omega$. This can be derived from the Hadamard
formula, see \cite{SokZol92}.

To obtain necessary optimality conditions, one sets the
Eulerian derivative $dJ(\Omega)[V]$ to zero.
For sufficient optimality conditions in conjunction with Eulerian
derivatives of second order (shape Hessian), we refer to
\cite{Epp07,Eppler_Unger97}.

In Section \ref{section_sensitivity_analysis_fatigue_Design}, we
will consider the shape derivative $dJ(\Omega)[V]$ of the cost
functional of the local Weibull model for LCF and will describe the so-called
adjoint method. The next sections address the question of the
existence of optimal shapes in case of linear elasticity.

\subsection{Sensitivity Analysis in Mechanical Design for Probabilistic
Fatigue}\label{section_sensitivity_analysis_fatigue_Design}

In this section, we outline sensitivity analysis in mechanical
design for fatigue, following \cite{Schmitz}. Sensitivity
analysis can provide necessary conditions for the existence of
optimal shapes which often leads to further insights into the
shape optimization problems. Moreover, sensitivity analysis can
establish a link to more efficient numerical optimization schemes
by use of the Eulerian derivative.

%\begin{remark}\label{lemma_existence_hypograph}
%Let $V\in\mathcal{V}^{k,\textrm{ad}}$ and $\Phi_t[V]$ be the
%one parameter family generated by $V$. Let the constants
%$L_1,L_2,L_3$ of Definition
%\ref{definition_design_variables_admissible_domains} be given and
%let $\phi\in U^\textrm{ad}=U^\textrm{ad}(L_1,L_2,L_3)$ be a
%$C^k$-admissible design variable. Then, there exists constants
%$L_2',L_3'$, a continuous function $L_1(t)$ and $\varepsilon>0$
%such that
%$\Phi_t[V](\Omega(\phi))=\Omega(\tilde{\phi}_t)$ for
%$0\leq t<\varepsilon$ and for some $\tilde{\phi}_t\in
%U^{\textrm{ad}}(L_1'(t),L_2',L_3')$.
%%When restricting to $V$ so that $\int_{\Omega(\phi)}\nabla
%%Vdx=0$ one can show $\frac{d}{dt}L_1'(t)|_{t=0}=0$, hence the
%%volume constraint $L_1'(t)=L_1$ is preserved to first order,
%%see \cite{SokZol92}. %Section 2.11 in \cite{SokZol92}.
%\end{remark}

State problems \eqref{eqa:ElasticityStrong} are now considered on
$C^{4,\phi}$-admissible
shapes $\Omega\in\O$. %\footnote{$C^4$ is needed for Theorem 6.2 in
%\cite{Gottschalk_Schmitz_arXiv}. Therefore, set $k=4$ in the definition
%of $\mathcal{O}$.}
Let $u_t=u(\Omega_t)$ denote state solutions of
on the domains
$\Omega_t=\Phi_t[V](\Omega)$. Employing their
regularity and boundedness properties along with the fact that
$u\circ \Phi_t$ defines a $C^{3,\phi}$-Banach isomorphism
from $[C^{3,\phi}(\Omega_t)]^3$ to $[C^{3,\phi}(\Omega)]^3$ shows
the following statement:

\begin{lem} 
Given the state problem \eqref{eqa:ElasticityStrong}, let
$f\in[C^{1,\phi}(\overline{\Omega^{\textrm{\emph{ext}}}})]^3$ and
$g\in[C^{2,\phi}(\overline{\Omega^{\textrm{\emph{ext}}}})]^3$ for
$\phi\in(0,1)$ %such that $g_i|_{\partial\Omega(\phi)_N}\in
%C^{2,\phi}(\partial\Omega(\phi)_N)$ for $\phi\in
%U^{\textrm{ad}}$ and
such that their corresponding norms are uniformly bounded by a
constant $k_{fg}$. Let $\Phi_t=\Phi_t[V]$ with
$V\in\mathcal{V}^{4}(\Omega^{\rm ext})$ and
$\Omega_t=\Phi_t(\Omega)$ with
$\Omega\in\O$. Furthermore,
suppose that $u_t$ solves  \eqref{eqa:ElasticityStrong} on $\Omega_t$.
Then, the pull-back solution $u^t=u_t\circ \Phi_t$ is
uniformly bounded in $[C^{3,\phi}(\Omega)]^3$ for $t$ sufficiently
small and depends continuously (in the $C^{3,\varphi}$-topology, $\varphi<\phi$) on $t$.
\end{lem}
 Note that $H^1$-continuity itself
is not sufficient to control the limits for material and shape
derivatives.
\begin{proof}
The uniform regularity follow from the results of Section 4. 
 One can show that $H^1$-continuity together with compactness in $C^{3,\varphi}$ implies
$C^{3,\varphi}$-continuity in our setting \cite{Schmitz}. 

Considering the weak
formulation of the mixed problem \eqref{eqa:ElasticityStrong} of linear
elasticity, we  obtain
continuity of pull-backed solutions $u^t=u_t\circ \Phi_t$
with $u_t=u(\Omega_t)$ and
$\Omega_t=\Phi_t[V](\Omega)$ by the continuity of Bilinear forms and right hand sides in the weak formalism. 
\end{proof}

Rewriting the bilinear form $B=B_\Omega$ of the elasticity
equation %of Theorem \ref{theorem_existence_weak_solution}
\begin{equation}\label{sensitivity.1.2.2}
    B_\Omega(u,v)=B(u,v)=\int_{\Omega}\left(\lambda\,(\nabla\cdot u
    \nabla\cdot
    v)+\frac{\mu}{2}\textrm{tr}[(\nabla u+\nabla u^T)(\nabla v+\nabla
    v^T)]\right)dx,
\end{equation}
let $B^t=B(\Omega,\nu,t)$ denote the pull-back of the bilinear
form %\footnote{For following technical steps in the this section,
%bear in mind that pull-back solutions can be arguments of
%pull-backed bilinear forms such as $B^t$.}
$B_t=B_{\Omega_t}$ to $\Omega$:
\begin{equation}
    B^t(u,v)=B_{\Omega_t}(u\circ \Phi_t^{-1},v\circ \Phi_t^{-1})
\end{equation}
for functions $u,v\in[H^1_{\partial\Omega_D}(\Omega)]^3$. Also denoting the linear
form $L$ %of Theorem \ref{theorem_existence_weak_solution}
by $L_\Omega$, let $L^t=L(\Omega,\nu,t)$ be the pull-back of the
linear form $L_t=L_{\Omega_t}$ to $\Omega$:
\begin{equation}
    L^t(v)=L_{\Omega_t}(v\circ \Phi_t^{-1})
\end{equation}
for $v\in[H^1_{\partial\Omega_D}(\Omega)]^3$.

%It follows from\footnote{Recall
%$F:G=\sum_{i=1}^n\sum_{j=1}^mF_{ij}G_{ij}$ and the relation
%$F:G=\textrm{tr}(FG^T)$ for matrices $F,G\in\mathbb{R}^{m\times
%n}$. This leads to $\nabla\cdot(g\circ
%h)(x)=\sum_{i=1}^3\frac{\partial}{\partial x_i}(g\circ
%h)_i(x)=\sum_{i=1}^3\sum_{j=1}^3\frac{\partial g_i(h(x))}{\partial
%y_j}\frac{\partial h_j(x)}{\partial x_i}=\nabla g(h(x)):\nabla
%h(x)$ for differentiable functions
%$g,h:D\subset\mathbb{R}^3\rightarrow\mathbb{R}^3$ with $D$ open.}
%$\nabla\cdot \left(u\circ \Phi_t^{-1}\right)(\cdot)=\nabla
%u(\Phi_t^{-1}(\cdot)):\nabla \Phi_t^{-1}(\cdot)=\textrm{tr}[\nabla
%u(\Phi_t^{-1}(\cdot))\nabla \Phi_t^{-1}(\cdot)^T]$ and from the
%substitution for multiple variables:
%\begin{align}\label{sensitivity.1.3}
%\begin{split}
%B^t(u,v)&=B_{\Omega_t}(u\circ \Phi_t^{-1},v\circ \Phi_t^{-1})\\
%&=\int_{\Omega}\left\{\det(\nabla
%\Phi_t)\left(\lambda\,\textrm{tr}[(\nabla u(\nabla \Phi_t^{-1}\circ
%\Phi_t)^T)]\textrm{tr}[(\nabla
%    v(\nabla \Phi_t^{-1}\circ \Phi_t)^T)]\right.\right.\\
%    &\quad\quad\left.+\frac{\mu}{2}\textrm{tr}[(\nabla u (\nabla \Phi_t^{-1}\circ \Phi_t)+(\nabla \Phi_t^{-1}\circ \Phi_t)^T\nabla
%    u^T)\right.\\
%&\hspace{18mm}\left.(\nabla v(\nabla \Phi_t^{-1}\circ \Phi_t)+(\nabla
%\Phi_t^{-1}\circ \Phi_t)^T\nabla
%    v^T)]\right\}dx.
%\end{split}
%\end{align}

%...obiges ist verschieden zu Hannos Gleichungen...???

%\begin{remark}\label{lemma_coercive_B_t}
Cauchy-Schwarz and Korn's second inequality can be applied to show
that $(B^t)_{t\in[-\varepsilon,\varepsilon]}$ is 
coercive on $[H^1_{\partial\Omega_D}(\Omega)]^3$ with uniform ellipticity constant. Moreover, the
mapping
\begin{equation}
t\in[-\varepsilon,\varepsilon]\mapsto B^t(u,\cdot)\in
([H^1_{\partial\Omega_D}(\Omega)]^3)'
\end{equation}
is continuous in the strong topology on
$([H^1_{\partial\Omega_D}(\Omega)]^3)'$ for all $u\in
[H^1_{\partial\Omega_D}(\Omega)]^3$.
The pull-back of the loads is defined according to Section 3.5 in
\cite{SokZol92}:

\begin{defn}\textbf{(Pull Back Loads)}\label{definition_pull_back_loads}$  $\\
Let $\gamma(t)=\det(\nabla \Phi_t)$ and
$\omega(t)=\gamma(t)\left\|\nabla \Phi_t^{-1}\nu\right\|$ on
$\partial\Omega_{N}$ with $\nu$ outward normal and
$\omega(t)=1$ on
$\partial\Omega\setminus\partial\Omega_{N}$. Let
$f:\mathbb{R}^3\rightarrow\mathbb{R}^3$ and
$g:\mathbb{R}^3\rightarrow\mathbb{R}^3$ be the volume and surface
loads on $\Omega_t$ and $\partial\Omega_t$, respectively. Then,
define the pull-back of the loads as
\begin{align}
\begin{split}
f^t:\Omega\rightarrow\mathbb{R}^3\quad\textrm{and}\quad f^t=\gamma(t)\,f\circ \Phi_t,\\
g^t:\partial\Omega\rightarrow\mathbb{R}^3\quad\textrm{and}\quad
g^t=\omega(t)\,g\circ \Phi_t.
\end{split}
\end{align}
\end{defn}

If $f\in [C^{1,\phi}(\mathbb{R}^3)]^3$ and $g\in
[C^{2,\phi}(\mathbb{R}^3)]^3$, the mapping
$[-\varepsilon,\varepsilon]\rightarrow
([H^1_{\partial\Omega_D}(\Omega)]^3)',\,t\mapsto L^t$ given by
$L^t(v)=\int_\Omega f^tvdx+\int_{\partial\Omega}g^tvdA$ is
continuous.
%\begin{proof}
%Using the trace theorem and considering the regularity properties
%of $f,g$ and $\Omega$ one can show $L^t\in
%([H^1_{\partial\Omega_D}(\Omega)]^3)'$, see Section II.3 in
%\cite{Braess}. Considering the regularity of the solution of the
%autonomous system\footnote{Confer Section 3 in \cite{Aulbach} for
%background on properties of autonomous systems.}
%(\ref{sensitivity.1.1}) leads to the continuity of the mapping
%$t\in[-\varepsilon,\varepsilon], t\mapsto L^t$ for some
%$\varepsilon>0$.
%
%\textbf{Achtung: Hanno schlaegt bei einem ausfuehlicheren Beweis vor:
%$\omega(t)=\omega(s)\omega(t-s)$ ... ??? beachte Eigenschaften
%autonomer Systeme. Es gelten aber etwas andere Gln. siehe
%schriftl. Unterlagen dazu...}
%
%$\quad$ \qed
%\end{proof}
Because continuity and uniform ellipticity conditions of $B^t$ and
$L^t$ are satisfied, one can show via results of functional
analysis, see Propositions in Appendix B of \cite{Schmitz},
that for $\varepsilon>0$ sufficiently small, the mapping
$[-\varepsilon,\varepsilon]\longrightarrow
[H^{1}(\Omega)]^3,\, t\longmapsto u^t=u_t\circ
\Phi_t$ is continuous, from which the continuity of $t\to u^ t$ in the $C^ {3,\varphi}$ topology follows by a compactness argument.
%\end{remark}
%\begin{proof} By Propositions \ref{proposition_u_t_continuity}
%and \ref{theorem_continuous_pull_back} we only need to check the
%continuity and uniform ellipticity conditions of $B^t$ and $L^t$
%as defined in Definitions \ref{definition_pull_back_bilinear} and
%\ref{definition_pull_back_loads}. This has been done in Lemma
%\ref{lemma_coercive_B_t} and \ref{lemma_continuous_l_t}. \qed
%\end{proof}

\subsection{Material and Shape Derivative of the Displacement Field}

The material and shape derivative of the displacement field $u$ is
now discussed, see also Section 3.5 in \cite{SokZol92}.
First, some additional operators and properties have to be
derived. Having defined $\varepsilon'(u)=-\frac{1}{2}\left(\nabla
u \nabla V+\nabla V^T\nabla u^T\right)$ for $u\in [H^1_{\partial\Omega_D}(\Omega)]^3$
and $V\in [C^k(\mathbb{R}^3)]^3$, we set
\begin{align}\label{sensitivity.2.1}
B'(u,v)=\int_\Omega\textrm{tr}[\varepsilon'(u)\sigma(v)+\sigma(u)\varepsilon'(v)+\nabla\cdot
V\sigma(u)\varepsilon(v)]dx.
\end{align}
$B'$ is the derivative of the bilinear form $B$ under the
deformation $\Omega\mapsto \Phi_t[V](\Omega)=\Omega_t$.
Here, we assumed that the Lame coefficients are constant.
Furthermore, let $u\in [H^2(\Omega)]^3,v\in [H^1_{\partial\Omega_D}(\Omega)]^3$ and
$k\geq2$ for $V\in [C^k(\mathbb{R}^3)]^3$. Then, we set
\begin{align}
\begin{split}
\label{eqa:RHSqdot}
&f''(u):=\left[\nabla\cdot (\nabla V\sigma(u))^T+\lambda\nabla
\textrm{tr}[\varepsilon'(u)]+2\mu\, \nabla\cdot\,\varepsilon'(u)
+\nabla\cdot(\nabla\cdot V\sigma(u))\right],\\
&g''(u):=\left[(\nabla V\sigma(u))^T\nu
+\lambda\textrm{tr}[\varepsilon'(u)]\nu+2\mu\varepsilon'(u)\nu+\nabla\cdot
V\sigma(u)\nu\right]
\end{split}
\end{align}
which leads to
 the following result:

%\begin{remark}\label{proposition_q_t_regularity}
\begin{prop}
For $t\in[0,\varepsilon]$ let $B_{\Omega_t}$ and $B'_t$ be
the bilinear forms defined in (\ref{sensitivity.1.2.2}) and
(\ref{sensitivity.2.1}), respectively, with $\Omega$ replaced by
$\Omega_t$. Let $f\in
[C^{2,\phi}(\overline{\Omega^{\textrm{\emph{ext}}}})]^3$ and $g\in
[C^{3,\phi}(\overline{\Omega^{\textrm{\emph{ext}}}})]^3$ for
$\phi\in(0,1)$. Moreover, let $q_t\in
[H_{\partial\Omega_D}^1(\Omega_t)]^3$ be the solution to the
following elasticity PDE in its weak form
\begin{align}\label{sensitivity.3.1}
B_{\Omega_t}(q_t,v)=\int_{\Omega_t}f'\cdot
v\,dx+\int_{\partial\Omega_{t_N}}g'\cdot
v\,dA-B_t'(u_t,v)\,\,\,\forall v\in
[H_{\partial\Omega_D}^1(\Omega_t)]^3
\end{align}
with $u_t$ state solution of \eqref{eqa:ElasticityStrong} on $\Omega_t$
and $f'=(\nabla\cdot V) f+\nabla fV,\,g'=(\nabla_{\partial\Omega}\cdot
V)g+(\nabla_{\partial\Omega} g)^TV$. Then, if $V\in
\mathcal{V}^{k,\textrm{ad}}$ for $k\geq4$ and $\varepsilon>0$ sufficiently
small, $q_t\in [C^{3,\phi}(\Omega_t)]^3$ and $
\|q_t\|_{[C^{3,\phi}(\Omega_t)]^3}<C$ holds where $C$ is
independent of $t$.
\end{prop}
%\end{remark}

The major steps for the derivation of this statement is the
application of Schauder estimates, see \cite{Hanno,Schmitz}, and the realization that equation
(\ref{sensitivity.3.1}) can be rewritten as follows:
\begin{equation}\label{sensitivity.4.1}
B_{\Omega_t}(q_t,v)=\int_{\Omega_t}\tilde{f}_t\cdot
v\,dx+\int_{\partial\Omega_{t_N}}\tilde{g}_t\cdot
v\,dA\quad\forall v\in [H_{\partial\Omega_D}^1(\Omega_t)]^3
%\quad\forall v\in \{w\in H^1(\Omega_t)|w|_{\partial\Omega_D}=0\}.
\end{equation}
with $\tilde{f}_t=f'+f''_t$ and $\tilde{g}_t=g'-g''_t$. Here,
$f''_t$ and $g''_t$ are defined as for
(\ref{eqa:RHSqdot}) with $\Omega$ replaced by
$\Omega_t$ and $u$ by $u_t$.
Then, one can show, similar as for the function $t\mapsto
u^t=u_t\circ \Phi_t[V]$, that the mapping $t\mapsto
q^t=q_t\circ \Phi_t[V]\in [C^{3,\varphi}(\Omega)]^3$ with
$t\in[0,\varepsilon]$ and $1>\phi>\varphi>0$ is continuous.
%\end{remark}
%
%\begin{proof}
%Note that $q_t$ satisfies equation (\ref{sensitivity.4.1}):
%\begin{equation}
%B_{\Omega_t}(q_t,v)=\int_{\Omega_t}\tilde{f}_t\cdot
%v\,dx+\int_{\partial\Omega_{t_N}}\tilde{g}_t\cdot
%v\,dA\quad\forall v\in [H_{\partial\Omega_D}^1(\Omega_t)]^3,
%%\quad\forall v\in \{w\in H^1(\Omega_t)|w|_{\partial\Omega_D}=0\}.
%\end{equation}
%see proof of Proposition \ref{proposition_q_t_regularity}.
%Considering the properties of $\Phi_t$ and substitution for multiple
%variables we obtain from (\ref{sensitivity.4.1})
%\begin{equation}
%B_{\Omega_t}(q^t\circ \Phi_t^{-1},v)=\tilde{L}_t(v)\quad\forall v\in
%[H_{\partial\Omega_D}^1(\Omega_t)]^3
%\end{equation}
%for a corresponding linear form $\tilde{L}_t$. As this weak
%formulation is equivalent to
%\begin{equation}
%B_{\Omega_t}(q^t\circ \Phi_t^{-1},v\circ \Phi_t^{-1})=\tilde{L}_t(v\circ
%\Phi_t^{-1})\quad\forall v\in [H_{\partial\Omega_D}^1(\Omega_t)]^3,
%\end{equation}
%we can argue as in the proof for Theorem
%\ref{theorem_continuous_pull_back_u_t} (note that $\tilde{L}_t$
%satisfies the required properties) which finally shows the
%continuity of $q^t=q_t\circ \Phi_t$. $\qed$
%\end{proof}
%
Using functional analytical arguments, one can finally motivate
that $q^t$ is the strong material derivative of $u^t$,
$t\in[0,\varepsilon]$.
Note that $t\mapsto \dot{u}^t\in [C^{3,\varphi}(\Omega)]^3$ with
$\phi>\varphi>0$ is also a continuous map, and that the shape
derivative is then given by $u'^t=\dot{u}^t-\nabla u^t\cdot V(t)$.

Finally, one can consider the shape differentiability and the
Eulerian derivative of the cost functional of the local and
probabilistic model for LCF. The cost functional is given by
%$J(\Omega,u)=J_{\textrm{vol}}(\Omega,u)+J_{\textrm{sur}}(\Omega,u)$
%\begin{align}\label{sensitivity.9.2}
%J_{\textrm{vol}}(\Omega,u)=\int_\Omega
%\mathcal{F}_{\textrm{vol}}(x,u(x),\nabla u(x))\,dx,\,
%J_{\textrm{sur}}(\Omega,u) =\int_{\partial\Omega}
%\mathcal{F}_{\textrm{sur}}(\cdot,u(\cdot),\nabla u(\cdot))\,dA
%\end{align}
%\begin{align*}
%dJ(\Omega)[V]&=\int_{\Omega} \frac{\partial}{\partial
%y}\mathcal{F}_{\textrm{vol}}(x,u(\Omega,V),\nabla
%u(\Omega,V))\cdot u'(\Omega,V)\,dx\\
%&\quad+\int_{\Omega}
%\nabla_q\mathcal{F}_{\textrm{vol}}(x,u(\Omega,V),\nabla
%u(\Omega,V)):\nabla
%u'(\Omega,V)\,dx\\
%&\quad+\int_{\partial\Omega}\mathcal{F}_{\textrm{vol}}(x,u(\Omega,V),\nabla
%u(\Omega,V)) \,(V(x)\cdot\nu(x))\,dA\\
%&\quad+\int_{\partial\Omega} \frac{\partial}{\partial
%y}\mathcal{F}_{\textrm{sur}}(x,u(\Omega,V),\nabla u(\Omega,V))\cdot\left(u'(\Omega,V)+\frac{\partial u}{\partial \nu}(\Omega,V)\,(V(x)\cdot\nu(x))\right)\,dA\\
%&\quad+\int_{\partial\Omega}
%\nabla_q\mathcal{F}_{\textrm{sur}}(x,u(\Omega,V),\nabla
%u(\Omega,V)):\nabla
%u'(\Omega,V)\,dA\\
%&\quad+\int_{\partial\Omega}\left(\nabla_q\mathcal{F}_{\textrm{sur}}(x,u(\Omega,V),\nabla
%u(\Omega,V)))\,\vdots\,
%\nabla^2u(\Omega,V)\cdot\nu(x)\right)\,(V(x)\cdot\nu(x))\,dA\\
%&\quad+\int_{\partial\Omega}\kappa(x)\mathcal{F}_{\textrm{sur}}(x,u(\Omega,V),\nabla
%u(\Omega,V)) \,(V(x)\cdot\nu(x))\,dA
%\end{align*}
%\end{remark}
%\begin{proof}
%As Theorem \ref{theorem_material_derivative} ensures the existence
%of the shape derivative of $u(\Omega)$ which is sufficiently
%regular, we can directly use the result (3.15) in
%\cite{SokZol92}. \qed
%\end{proof}
\begin{align}\label{sensitivity.7.2}
\begin{split}
J(\Omega,u)=\int_{\partial\Omega}{\cal F}_{\rm sur}(\nabla u)\,
dA=\int_{\partial\Omega}\left(\frac{t}{N_{i_{\textrm{det}}}(\nabla
u)}\right)^m\, dA
\end{split}
\end{align}
with%\footnote{Recall that $[M]_{\textrm{vM}}$ is the von Mises
%stress value of a matrix $M\in\mathbb{R}^{3\times3}$. We denoted
%the case $[\sigma^e(u)]_{\textrm{vM}}$ for the elastic stress
%tensor field $\sigma^e(u)$ of Definition
%\ref{definition_mixed_problem_thermoelasticity} also by
%$\sigma_{\textrm{vM}}^e(u)$.}
\begin{equation}
N_{i_{\textrm{det}}}(M)=\varphi(\left[\lambda \, {\rm
tr}(M)I+\mu(M+M^T)\right]_{\textrm{vM}})\textrm{ for }
M\in\mathbb{R}^{3\times3},\quad\varphi=\textrm{CMB}^{-1}\circ\textrm{RO}\circ\textrm{SD}.
\end{equation}
%for matrices $M\in\mathbb{R}^{3\times3}$
Here, $[M]_{\textrm{vM}}$ is the von Mises stress value of a
matrix $M\in\mathbb{R}^{3\times3}$, and $\textrm{CMB, RO, SD}$ are
material and lifing functions. See \cite{Schmitz} for more
background. Following Section 3.3 in \cite{SokZol92} the
Eulerian derivative is given by\footnote{In
(\ref{sensitivity.8.1}), the Hessian $\nabla^2u$ is a
three-dimensional matrix with one index regarding the components
of $u$ which contracts with the index of the partial derivatives
of $\nabla\mathcal{F}_{\textrm{sur}}$, and the other two indices
with respect to the partial derivatives which contract with the
remaining index of $\nabla\mathcal{F}_{\textrm{sur}}$ and with
$\nu$.}
\begin{align}\label{sensitivity.8.1}
\begin{split}
dJ(\Omega)[V]&=\int_{\partial\Omega}
\nabla\mathcal{F}_{\textrm{sur}}(\nabla u):\nabla
u'(\Omega,V)\,dA\\
&\quad+\int_{\partial\Omega_N}\left(\nabla\mathcal{F}_{\textrm{sur}}(\nabla
u))\,:\,
\nabla^2u\cdot\nu(x)\right)\,(V(x)\cdot\nu(x))\,dA\\
&\quad+\int_{\partial\Omega_N}\bar{\kappa}\,\mathcal{F}_{\textrm{sur}}(\nabla
u) \,(V(x)\cdot\nu(x))\,dA
\end{split}
\end{align}
with $\bar{\kappa}$ the mean curvature of $\partial\Omega$ and
with\footnote{Here, $\nabla[M]_{\textrm{vM}}$ denotes the gradient
of the von Mises stress at the value of a matrix
$M\in\mathbb{R}^{3\times3}$.}
\begin{align}
\begin{split}
(\nabla\mathcal{F}_{\textrm{sur}}(\nabla
u))_{kl}=&m(\textrm{CMB}^{-1})(\textrm{RO}(\textrm{SD}(\sigma_{a}^{\rm el}(u))))^{m-1}(\textrm{CMB}^{-1})'(\textrm{RO}(\textrm{SD}(\sigma_{a}^{\rm el}(u))))\\&\cdot\textrm{RO}'(\textrm{SD}(\sigma_{a}^{\rm el}(u)))\cdot\textrm{SD}'(\sigma_{a}^{\rm el}(u))\cdot\nabla[\sigma(u)]_{\textrm{vM}}\cdot(\lambda\delta_{kl}
I+\mu(E_{kl}+E_{lk})) ,%_{k,l},%\quad k,l=1,2,3,
\end{split}
\end{align}
where $(E_{kl})_{kl}\in\mathbb{R}^{3\times3}$ is everywhere zero
except for the $(k,l)$-th component which is one. Sufficiently
regular lifing and material functions $\textrm{CMB, RO, SD}$
ensure the existence of the derivatives of $\textrm{CMB}^{-1}$.
Under the assumption that the material derivative
$u'=u'(\Omega,V)$ exists and is sufficiently regular, the previous
results constitute an important advantage of the local and
probabilistic model for LCF compared to usual deterministic models
with infimum functions, where insufficient regularity is given due
to the non-differentiability of the infimum function.

\subsection{Continuous Adjoint Equation and 1${}^{\rm st}$ Order Optimality Conditions}

As the computation of the shape derivative $u'(\Omega,V)$ can be
numerically very costly, the so-called adjoint method is often
applied for numerical optimization schemes, consider Section 3.1
in \cite{ShapeOpt} and
\cite{SokZol92}. This method uses the shape derivative
$u'(\Omega,V)$ to rewrite
the Eulerian derivative of the cost functional.  %, and are
%less mathematically formal than in the previous sections because
%of open regularity questions which are subject of current
%research, see \cite{Gottschalk_Schmitz_Krause}.
%\textbf{Adjoint Method}
We now apply the adjoint method to cost functional
(\ref{sensitivity.7.2}) in the setting of 
with $C^{k,\phi}$-admissible shapes. Using the results of Theorem 3.11 in
\cite{SokZol92}, where a linear elastic boundary value problem describes
the shape derivative, the following holds: For the linear elastic
state problem \eqref{eqa:ElasticityStrong}, the material derivative
$u'(\Omega,V)$ is determined by
\begin{align}\label{sensitivity.11.1}
\begin{split}
&\nabla\cdot\sigma(u')=0\hspace{69mm}\textrm{in }\Omega,\\
&u'=0\hspace{82mm}\textrm{on }\partial\Omega_D,\\
&\sigma^e(u')\cdot \nu(x)=(V(x)\cdot\nu(x))\,f+(V(x)\cdot\nu(x))\,\bar{\kappa}\, g\\
&\hspace{24mm}-\nabla_{\partial\Omega}\cdot((V(x)\cdot\nu(x))\sigma_\tau)\hspace{24mm}\textrm{on
}\partial\Omega_N,
\end{split}
\end{align}
with
$\sigma_\tau=\sigma_\tau(u)=\sigma(u)\nu-(\nu\sigma(u)\nu)\nu$.
The corresponding weak formulation is given by
\begin{align}\label{sensitivity.12.1}
\begin{split}
&B(u',v)=l(V,v)\quad\forall v\in[H^1_{\partial\Omega_D}(\Omega)]^3,\\
B(w,v)&=\int_\Omega\lambda(\nabla\cdot w) (\nabla\cdot v)\,\,dx+\int_\Omega 2\mu\,{\rm tr}(\varepsilon(w)\varepsilon(v))dx,\quad w,v\in[H^1_{\partial\Omega_D}(\Omega)]^3,\\
l(V,v)&=\int_{\partial\Omega_{\textrm{N}}}\left[(V(x)\cdot\nu(x))\,f+(V(x)\cdot\nu(x))\,\bar{\kappa}\, g-\nabla_{\partial\Omega}\cdot((V(x)\cdot\nu(x))\sigma_\tau)\right]\,v\,dA\\
&=\int_{\partial\Omega_{\textrm{N}}}(V(x)\cdot\nu(x))\left[(f+\bar{\kappa}\,g)v+\sigma_\tau(\nabla_{\partial\Omega}
v)\right]\,dA,\quad v\in[H^1_{\partial\Omega_D}(\Omega)]^3.
\end{split}
\end{align}
Then, the adjoint state $p$ is defined via the adjoint problem
\begin{align}\label{sensitivity.13.1}
\begin{split}
B( v,p)&=\int_{\partial\Omega}
\nabla\mathcal{F}_{\textrm{sur}}(\nabla u):\nabla
 v\,dA\quad\forall v\in[H^1_{\partial\Omega_D}(\Omega)]^3.
\end{split}
\end{align}
Due to (\ref{sensitivity.12.1}) the relationship
$l(V,p)=B(u',p)=\int_{\partial\Omega}
\nabla\mathcal{F}_{\textrm{sur}}(\nabla u):\nabla u'\,dA$ holds if
the adjoint state $p$ exists in $V_{\textrm{DN}}$. Thus, the
adjoint method applied to the local and probabilistic model for
LCF given in (\ref{sensitivity.8.1}) results in the Eulerian
derivative
\begin{align}\label{sensitivity.14.1}
\begin{split}
dJ(\Omega)[V]&=l(V,p)+\int_{\partial\Omega_N}\left(\nabla\mathcal{F}_{\textrm{sur}}(\nabla
u)):
\nabla^2u\cdot\nu(x)\right)\,(V(x)\cdot\nu(x))\,dA\\
&\quad\quad\quad\quad+\int_{\partial\Omega_N}\bar{\kappa}\,\mathcal{F}_{\textrm{sur}}(\nabla
u) \,(V(x)\cdot\nu(x))\,dA\\
&=\int_{\partial\Omega_N}\left[(f+\bar{\kappa}\,g)\,p+\sigma_\tau(\nabla_{\partial\Omega}
p)+\nabla\mathcal{F}_{\textrm{sur}}(\nabla u)):
\nabla^2u\cdot\nu(x)+\bar{\kappa}\,\mathcal{F}_{\textrm{sur}}(\nabla
u)\right]\\
&\hspace{15mm}\,\,(V(x)\cdot\nu(x))\,dA.
\end{split}
\end{align}
This expression for the Eulerian derivative in fact is in Hadamard form. It can be exploited in a
gradient-based optimization scheme, as described at the end of
Section 4.3 in \cite{Schmitz}, for example. Note that the
regularity of the right-hand side of (\ref{sensitivity.13.1}) has
to be further discussed as $\nabla v$ is evaluated on the
surface $\partial\Omega$. Possibly the use of distribution theory
leads to an appropriate regularization technique to resolve this
issue. Furthermore note that the existence proof of the material
and shape derivatives still has to be finished in a mathematically
rigorous way.

Necessary optimality conditions can now also be given by setting
the Eulerian derivative (\ref{sensitivity.14.1}) to zero, i.e.
$dJ(\Omega)[V]=0$. Thereby, an optimal shape $\Omega^*$ has to
fulfill
\begin{align}\label{sensitivity.15.1}
\begin{split}
(f+\bar{\kappa}\,g)\,p+\sigma_\tau(\nabla_{\partial\Omega}
p)+\nabla\mathcal{F}_{\textrm{sur}}(\nabla u)):
\nabla^2u\cdot\nu(x)+\bar{\kappa}\,\mathcal{F}_{\textrm{sur}}(\nabla
u)=C\quad\textrm{on }\partial\Omega^*_{\textrm{N}}
\end{split}
\end{align}
for some constant $C$. In addition, the state and adjoint problem
for $u$ and $p$, respectively, yields the final BVPs that describe
the necessary optimality conditions for an optimal
$C^4$-admissible shape $\Omega^*$.

The previous approach leads to an efficient gradient-based
optimization scheme for mechanical design in fatigue and shows the
importance of a mathematical approach to shape optimization. In
multi-physical design environments which include fatigue design,
this can be very important. Note that in computational fluid
dynamics (CFD) gradient-based shape optimization is established,
see \cite{Schmidt}, and would benefit from an equally fast
optimization approach on the structural integrity side.
gasturbine blades are examples for design approaches where
efficiency is a decisive criteria and where complicated structural
integrity issues are present due to fatigue.

%%%%%%%%%%%%%%%%%%%%%%%%%%%%%%%%%%%%%%%%%%%%%%%%%%%%%%%%%%%%%%%%%%%%%%%%%%%%%%%%%%%%
%%%%%%%%%%%%%%%%%%%%%%%%%%%%%%%%%%%%%%%%%%%%%%%%%%%%%%%%%%%%%%%%%%%%%%%%%%%%%%%%%%%%
%%%%%%%%%%%%%%%%%%%%%%%%%%%%%%%%%%%%%%%%%%%%%%%%%%%%%%%%%%%%%%%%%%%%%%%%%%%%%%%%%%%%
%%%%%%%%%%%%%%%%%%%%%%%%%%%%%%%%%%%%%%%%%%%%%%%%%%%%%%%%%%%%%%%%%%%%%%%%%%%%%%%%%%%%

\section{\label{sec:RiskManagement}Risk Management through Optimal Service Scheduling}

\subsection{Intentions of Service Scheduling}
This section presents an idea, how to create an optimal maintenance schedule for a gas turbine by using the knowledge about failure probabilities from the former sections. Maintenance plays an important role in gas turbine operation, because it increases the reliable operation. Next to the fuel consumption service cost are the second largest portion of life cycle cost in operation of a gas turbine. Therefore an optimized maintenance schedule will improve profitability of a gas turbine.

Today, predictive maintenance actions for a gas turbine are determined by one life counter $c\left(t\right )$ which represents the engine model with regards to the consumed life of the gas turbine. If the gas turbine reaches the deterministic life counter limit $l_\text{det}$ for the respective maintenance interval, then the operator has to carry out the respective service according to its maintenance manual. The drawback of having only a deterministic life counter limit is, that we cannot extend interval limits. To overcome this point, we establish the probabilistic nature of the different failure mechanism into the maintenance model. This gives us the option to extend service intervals or to skip service actions.

\subsection{Modeling the Expected Cash Flow}
Our aim is to maximize the excepted revenue of the gas turbine for the given time interval $\left[0,\infty\right)$. We provide only a very basic maintenance / reward model of the gas turbine to establish the basic ideas in our modeling approach. Therefore we analyze the operators cash flow in present value formulation which is given by
\begin{align}
\label{eq:pv}
\text{pv}\left(\tau\right) &= e^{-i_{\text{eff}}\tau} \cdot I \cdot \chi\left(\tau\right) _{\left\lbrace\text{gas turbine is not failed and not in service}\right\rbrace}  &  \text{(income)}\nonumber\\
&- e^{-i_{\text{eff}}\tau} \cdot C_\text{R} \cdot \chi\left(\tau\right) _{\left\lbrace\text{gas turbine fails}\right\rbrace}  & \text{(failure costs)}\\
&- e^{-i_{\text{eff}}\tau} \cdot C_\text{M} \cdot \chi\left(\tau\right) _{\left\lbrace\text{gas turbine is not failed and starts service}\right\rbrace}   & \text{(service costs)}\nonumber 
\end{align}
where $\chi\left(\tau\right) $ is an indicator function which includes the failure time $T_\text{F}$ of the gas turbine, $I\geq0$ is the operators revenue per unit of time, $C_\text{M} $ is the service fee, $C_\text{R}$  are the failure costs and $0<i_{\text{eff}}\leq1$ is the discount factor. We get the cumulative cash flow by integrating over $\tau$
\begin{align}
\text{PV}=\int_0^\infty\text{pv}\left(\tau\right) d\tau.
\end{align}
Since the failure time $\tau_\Omega$ of the gas turbine or gas turbine component is a random variable with the underlying survival probability $S_\Omega\left(\tau\right)$, we can calculate the excepted present value (EPV). The operator earns only money $I$ in (\ref{eq:pv}a) or pays $C_\text{M}$ for service at time $\tau$ in (\ref{eq:pv}c), if the gas turbine survives until $\tau$ with probability $S\left(\tau\right)$. But also there is a hazard  $h\left(\tau\right) = \frac{f\left(\tau\right)}{ S\left(\tau\right)}$ to pay the failure costs $C_\text{R}$ at time $\tau$ in (\ref{eq:pv}b) . $h\left(\tau\right)$ is the hazard rate and $f\left(\tau\right)$ is failure density function. Further, we split the interval $\left[0,\infty\right]$ into parts where the gas turbine is in operation or in service. Therefore we establish the service duration $W>0$ and the service start times $t_i$, $i=1,2,\ldots$. We define the gas turbine uptime as
\begin{align}
U:=\left[0, t_{1}  \right]\cup\bigcup_{i=i}^\infty \left[t_i+W, t_{i+1}  \right]
\end{align}
and in sum we get the expected net present value
\begin{align}
\text{EPV}=\int_U e^{-i_{\text{eff}}\tau} S\left(\tau\right) \left[ I -C_\text{R}   h\left(\tau\right) \right]d\tau-\sum^\infty_{i=1}e^{-i_{\text{eff}}t_i} S\left(t_i\right) C_\text{M}.
\end{align}
\subsection{Optimal Service Intervals}
We establish four assumptions to convert our EPV term into an optimization problem. First, we assume that our survival function $S\left(\tau\right)$ is a function of the actual life counter $c\left(t\right )$, e.g. $S\left(\tau\right)=S\left(c\left(t\right )\right)$ and therefore the hazard rate and failure density are also functions of the life counter. In addition we assume only cyclic damage mechanism like LCF and therefore the the counter $c\left(t\right)$ give us the cumulated and weighted number of cyclic events until $t$. As a second step we assume that there only exists one service action which sets the gas turbine in an as new state. Third, service is done in a periodic pattern. Our last assumption is that he operating regime of the gas turbine is constant, that means the gradient of the life counter  $c\left(\tau\right)$ is constant.\\
Due to the four assumption we have to determine an optimal interval length $\Delta>0$ between two service action. Further we can conclude new properties for our EPV term from the assumptions. We conclude from assumption two and three, that the hazard rate $h\left(\tau\right)$ is set to zero due to a service and we obtain
\begin{align}
h\left(\tau\right) = 0 \quad \text{for all} \ \tau\in\left[i \left(\Delta+ W\right), \left(i+ 1\right) \left(\Delta+ W\right) \right], \ i=1,2,3,\ldots,
\end{align}
during a service action and also we get
\begin{align}
\label{eq:hr}
h\left(\tau\right) =  h\left(\tau \mod i \left(\Delta+ W\right) \right), \quad \text{for all}  \ \tau, \ i=1,2,3,\ldots.
\end{align}
\begin{figure}[t]
	\centering
		\includegraphics[width=1.0\textwidth]{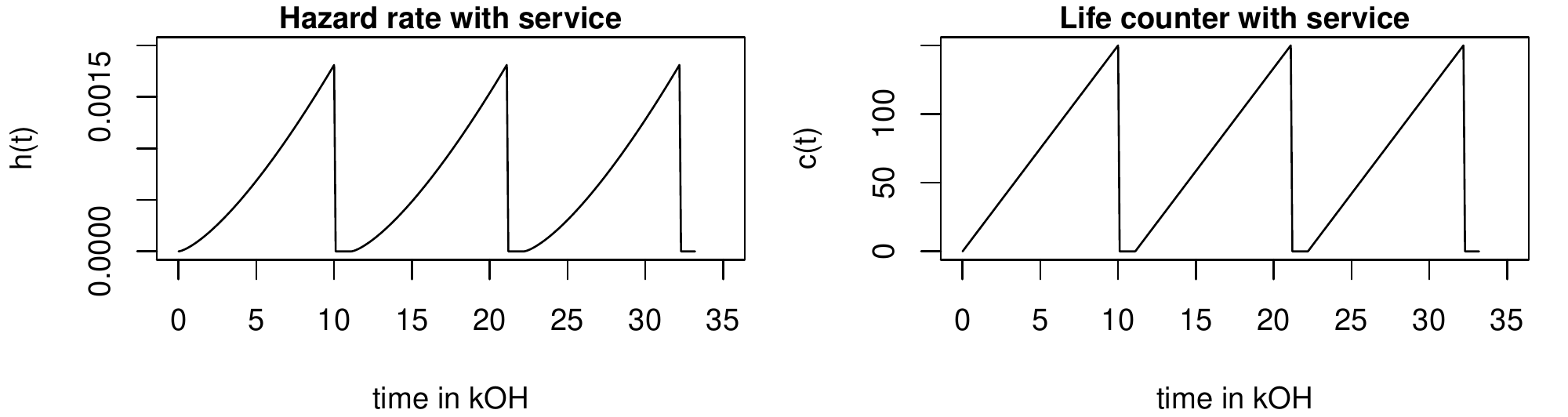}
   \caption{The left figure shows an idealized developing of the hazard rate function $h\left(t\right)$ and the right figure shows the developing of the associated life counter $c\left(t\right)$. In both figures the function are set to zero during an outage.}
  \label{fig:toymodel}
\end{figure}
Figure \ref{fig:toymodel} summarize this properties. For the cumulative hazard rate $H\left(\tau\right)$ we obtain from (\ref{eq:hr})  
\begin{align}
H\left(\tau\right) &=  \int_0^\tau  h\left(t\right) dt =  \int_0^\tau  h\left(t\mod i \left(\Delta+ W\right) \right) dt \notag \\
\label{eq:chr}&=\left\lfloor \frac{\tau}{\Delta + W}\right\rfloor H\left(\Delta\right) + H\left(\tau\mod i \left(\Delta+ W\right) \right)
\end{align}
with $\lfloor x \rfloor :=\max \left\lbrace  y \in\mathbb{Z} \ \vert \ y\leq x  \right\rbrace $. We note the relation
\begin{align}
\label{eq:relsh}
S\left(\tau\right) = \exp\left(- H\left(\tau\right)  \right)
\end{align}
between the survival function $S\left(\tau\right)$ and hazard rate $h\left(\tau\right)$  and we obtain from (\ref{eq:chr}) and (\ref{eq:relsh})
\begin{align}
\label{eq:sp}
S\left(\tau\right) =  S\left(\Delta\right)^{\left\lfloor \frac{\tau}{\Delta + W}\right\rfloor}    S\left(\tau \mod i \left(\Delta+ W\right) \right).
\end{align}
With the help of the assumptions, the service duration $W$, (\ref{eq:chr}) and (\ref{eq:sp}) we can simplify the EPV term and we get
\begin{align}
\begin{split}
\text{EPV}&=\int_U e^{-i_{\text{eff}}\tau} S\left(\tau\right) \left[ I -C_\text{R}   h\left(\tau\right) \right]d\tau-\sum^\infty_{i=1}e^{-i_{\text{eff}}t_i} S\left(t_i\right) C_\text{M}\\
&=\sum^\infty_{i=0} \int_{i\left(\Delta + W\right)}^{\left(i+ 1\right)\left(\Delta + W\right)} e^{-i_{\text{eff}}\tau} S\left(\tau\right) \left[ I - C_\text{R}  h\left(\tau\right)\right]d\tau\ \\
&- e^{-i_{\text{eff}}\left(\left(i+1\right)\Delta +i W\right)} S\left(\left(i+1\right)\Delta +i W\right) C_\text{M} \\
&= \sum^\infty_{i=0}  e^{-i_{\text{eff}} i \left( \Delta + W\right)}  S\left(\Delta\right)^i \int_0^\Delta e^{-i_{\text{eff}}\tau}S\left(\tau\right) \left[I - C_\text{R} h\left(\tau\right)\right]d\tau\\
& - e^{-i_{\text{eff}}\left(\left(i+1\right)\Delta +i W\right)} S\left(\Delta\right)^{i+1} C_\text{M}.
\end{split}
\end{align}
Finally, we split the EPV term into two geometric series and we get
\begin{align}
\text{EPV}\left(\Delta\right) = \frac{\int_0^\Delta e^{-i_{\text{eff}}\tau}S\left(\tau\right) \left[I - C_\text{R} h\left(\tau\right)\right]d\tau-C_\text{M} e^{-i_{\text{eff}}\Delta} S\left(\Delta\right) }{1-e^{-i_\text{eff}\left(\Delta + W\right)}S\left(\Delta\right)}
\end{align}
for our objective function. Our complete optimization problem is given by
\begin{align}
\begin{split}
\max_\Delta & \frac{\int_0^\Delta e^{-i_{\text{eff}}\tau}S\left(\tau\right) \left[I - C_\text{R} h\left(\tau\right)\right]d\tau-C_\text{M} e^{-i_{\text{eff}}\Delta} S\left(\Delta\right) }{1-e^{-i_\text{eff}\left(\Delta + W\right)}S\left(\Delta\right)}\\
\text{subject to} & \ \\
\Delta & \geq 0 .
\end{split}
\end{align}
As last step we present an example. We consider a Weibull distribution for the failure time $T_\text{F}$. We get 
\begin{align*}
h\left(\tau\right) = \frac{m}{\eta}\left(\frac{\tau}{\eta}\right)^{m-1} \quad \text{and} \quad S\left(\tau\right) =e^{-\left(\frac{\tau}{\eta}\right)^m}
\end{align*}
where $m$ is the shape parameter and $\eta$ is the scale parameter. We choose $\eta = 2000.0$, $m=2.4$, $I=50.0$, $C_\text{M}=300$, $C_\text{R}=500000.0$, $i_\text{eff}=0.003$ and $W=30.0$. The optimal interval $\Delta^\star$ is $153.0$ and the objective function value $\text{EPV}\left(\Delta^\star\right)$ is $12233.11$. 
\begin{figure}[t]
	\centering
		\includegraphics[width=1.0\textwidth]{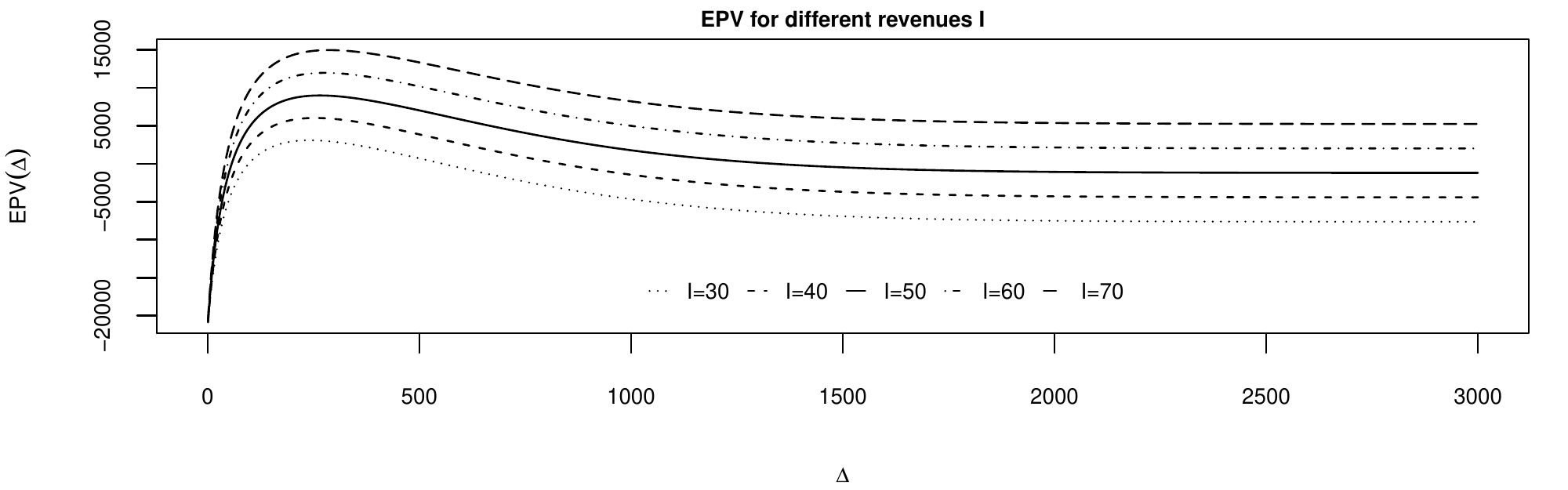}
	\caption{The figure shows the developing of the EPV over the interval length $\Delta$ for different revenue parameters $I$. The solid line belongs to the example from the text.}
	\label{fig:ana1c_ex}
\end{figure}
In figure \ref{fig:ana1c_ex} we plotted $\text{EPV}\left(\Delta\right)$ against $\Delta$ and the solid line presents our example.

\section{\label{sec:Conclusion}Conclusions and Outlook}
\label{sec:CO}

In the present article, we have given an overview of the manifold interconnections between failure time processes and shape optimization. While the field of probabilistic life calculation and the optimization of reliability is an urgent need in engineering, there are several mathematical implications that are interesting in their own right. At the same time, risk figures based on empirical or physics based laws can be used for a more rational decision making on the operation of mechanical components.
This work, rather than going into the detail of one of the ramifications of the topic, intends to present this circle of ideas as a whole (somewhat in the spirit of \cite{Pflug}, where also our title is borrowed from).   

At this point the question arises naturally, how this field will further develop. Naturally, the application to concrete engineering problems is most important.  If probabilistic methods can realize their potential in the more accurate and risk aware design for mechanical integrity, this will drive the future mathematical and algorithmic development.

Simultaneously  it is of interest that in this field the mathematical research does not follow the new development in engineering applications, but accompanies it. This is an opportunity to bring in advanced mathematical methods, like shape calculus, from the outset. 

A number of research topics  can be formulated on the basis of the investigations presented here: On the theoretical side, the mathematical construction of a flow towards optimal reliability on the infinite dimensional manifold of shapes \cite{Schulz} certainly is an interesting goal. 

On the numerical side, the implementation of shape derivatives in a discrete or continuous adjoint setting could lead to new design procedures based on reliability.

From a material science point of view, a further study of the multiscale models of Section \ref{sec:MultiScale} and their extension to multi-load and random load scenarios poses a new challenge, in particular what concerns the interplay between modelling and experimental validation.   Also, more damage mechanism, as e.g. creep or thermo-mechanical fatigue (TMF), have to be taken into account.

Last but not least, the impact of reliability calculations on service actions, in particular if more diverse service actions as "repair as new" are available, is an interesting topic for future research.
\vspace{1cm}

\noindent \textbf{Acknowledgements:} Hanno Gottschalk would like to thank Sergio Albeverio, Ana Bela Cruceiro and Derek Holms for their kind invitation to the CIB and the staff of CIB for their hospitality. Nadine Moch and Mohamed Saadi have been supported by AG Turbo Project 4.1.2 and 4.1.13 co financed by BMWi and Siemens Energy. We also thank T.\ Beck (TU Kaiserslautern), B.\ Beckmann, H.\ Harders, G.\ Rollmann and A.\ Sohail (Siemens Energy) for interesting discussion.

\bibliographystyle{plain}
\bibliography{quellen}

\begin{thebibliography}{10}

\bibitem{Agm59}
S.~Agmon, A.Douglis, and L.~Nirenberg.
\newblock Estimates near the boundary for solutions of elliptic partial
  differential equations satisfying general boundary conditions i.
\newblock {\em Communications On Pure And Applied Mathematics}, Vol.
  XII:623--727, 1959.

\bibitem{Agm64}
S.~Agmon, A.~Douglis, and L.~Nirenberg.
\newblock Estimates near the boundary for solutions of elliptic partial
  differential equations satisfying general boundary conditions ii.
\newblock {\em Communications On Pure And Applied Mathematics}, Vol.
  XVII:35--92, 1964.

\bibitem{Werkstoffe}
M.~B\"aker, H.~Harders, and J.~R\"osler.
\newblock {\em Mechanisches Verhalten der Werkstoffe}.
\newblock Vieweg+Teubner, 3rd edition, 2008.

\bibitem{BG}
L.~Bittner and H.~Gottschalk.
\newblock Optimal reliability for components under thermomechanical cyclic
  loading.
\newblock {\em arXiv}, 1601:00419, 2016.

\bibitem{BGS}
M.~Bolten, H.~Gottschalk, and S.~Schmitz.
\newblock Minimal failure probability for ceramic design via shape control.
\newblock {\em Journal of Optimization Theory and Applications}, 166:983--1001,
  2015.

\bibitem{Alum}
L.P. Borrego, L.M. Abreu, J.M. Costa, and J.M. Ferreira.
\newblock Analysis of low cycle fatigue in almgsi aluminium alloys.
\newblock {\em Engineering Failure Analysis}, 11:715--725, 2004.

\bibitem{BucBut05}
D.~Bucur and G.~Buttazzo.
\newblock {\em Variational Methods in Shape Optimization Problems}.
\newblock Birkhäuser, 2005.

\bibitem{Chen75}
D.~Chenais.
\newblock On the existence of a solution in a domain identification problem.
\newblock {\em Journal of Mathematical Analysis and Applications}, 52:189--289,
  1975.

\bibitem{Cia1988}
Phillipe Ciarlet.
\newblock {\em Mathematical Elasticity - Volume I: Three-Dimensional
  Elasticity}, volume~20 of {\em Studies in Mathematics and its Applications}.
\newblock North-Holland, 1988.

\bibitem{Epp07}
K.~Eppler.
\newblock Efficient shape optimization algorithms for elliptic boundary value
  problems.
\newblock Habilitation Thesis, Universtity of Chemnitz, March 5 2007.

\bibitem{Eppler_Unger97}
K.~Eppler and A.~Unger.
\newblock Boundary control of semilinear elliptic equations - existence of
  optimal solutions.
\newblock {\em Control and Cybernetics}, 26, No. 2:249--259, 1997.

\bibitem{ErnGuerm04}
A.~Ern and J.-L. Guermond.
\newblock {\em Therory and Practice of Finite Elements}.
\newblock Springer, New York, 2004.

\bibitem{EscobarMeeker}
L.~A. Escobar and W.~Q. Meeker.
\newblock {\em Reliability Statistics}.
\newblock Wiley, 1998.

\bibitem{Evans}
L.~C. Evans.
\newblock {\em Partial Differential Eqations}.
\newblock American Mathematical Society, 2. edition, 2010.

\bibitem{Fett}
D.~Fett and D.~Munz.
\newblock {\em Mechanische Eigenschaften von Keramik}.
\newblock Springer, 1989.

\bibitem{Fujii}
N.~Fujii.
\newblock Lower semicontinuity in domain optimization problems.
\newblock {\em Journal of Optimization Theory and Applications}, 59:pp
  407--422, December 1988.

\bibitem{GilbTrud}
D.~Gilbarg and N.~S. Trudinger.
\newblock {\em Elliptic Partial Differential Equations of Second Order}.
\newblock Springer, Berlin Heidelberg New York, 1977.

\bibitem{Hanno}
H.~Gottschalk and S.~Schmitz.
\newblock Optimal reliability in design for fatigue life.
\newblock {\em Siam Journal of Control and Optimization}, 52 (5):2727--2725,
  2015.

\bibitem{GSSKRB}
S.~Schmitz H.~Gottschalk, T.~Seibel, R.~Krause, , G.~Rollmann, and T.~Beck.
\newblock Probabilistic schmid factors and scatter of lcf life.
\newblock {\em Materials Science and Engineering}, to appear, 2014.

\bibitem{ShapeOpt}
J.~Haslinger and R.~A.~E. Mäkinen.
\newblock {\em Introduction to Shape Optimization}.
\newblock SIAM, 2003.

\bibitem{Vorm}
O.~Hertel and M.~Vormwald.
\newblock Statistical and geometrical size effects in notched mem- bers based
  on weakest-link and short-crack modelling.
\newblock {\em Engineering Fracture Me- chanics}, 95:72–83, 2012.

\bibitem{Kallenberg}
O.~Kallenberg.
\newblock {\em Random measures}.
\newblock Akademie Verlag, 1982.

\bibitem{Pflug}
G.~C. Pflug and W.~R\"omisch.
\newblock {\em Modeling, Measuring and Managing Risk}.
\newblock World Scientific, 2007.

\bibitem{Erm}
D.~Radaj and M.~Vormwald.
\newblock {\em Ermüdungsfestigkeit}.
\newblock Springer, Berlin Heidelberg New York, 3. edition, 2007.

\bibitem{Schmidt}
S.~Schmidt.
\newblock {\em Efficient Large Scale Aerodynamic Design Based on Shape
  Calculus}.
\newblock Dissertation, Universit\"at Trier, 2010.

\bibitem{Schmitz}
S.~Schmitz.
\newblock {\em A Local and Probabilistic Model for Low-Cycle Fatigue.: New
  Aspects of Structural Analysis}.
\newblock Hartung-Gorre, 2014.

\bibitem{SBKRSG}
S.~Schmitz, T.~Beck, R.~Krause, G.~Rollmann, T.~Seibel, and Hanno Gottschalk.
\newblock A probabilistic model for lcf.
\newblock {\em Computational Materials Science}, 79:584--590, 2013.

\bibitem{SSGBRK}
S.~Schmitz, T.~Seibel, H.~Gottschalk, T.~Beck, G.~Rollmann, and Rolf Krause.
\newblock Probabilistic analysis of the lcf crack initiation life for a turbine
  blade under thermo-mechanical loading.
\newblock {\em Proc. Int. Conf LCF 7}, 2013.

\bibitem{SSBGRK2}
S.~Schmitz, T.~Seibel, H.~Gottschalk, T.~Beck, G.~Rollmann, and Rolf Krause.
\newblock Risk estimation of lcf crack initiation.
\newblock {\em Proc. ASME Turbo Expo}, GT2013:94899, 2013.

\bibitem{Schulz}
V.~Schulz.
\newblock A riemannian view on shape optimization.
\newblock {\em Foundations of Computational Mathematics}, 14:483--501, 2014.

\bibitem{SokZol92}
J.~Sokolovski and J.-P. Zolesio.
\newblock {\em Introduction to Shape Optimization - Shape Sensitivity
  Analysis}.
\newblock Springer, Berlin Heidelberg, 1992.

\bibitem{Sornette}
D.~Sornette, T.~Magnin, and Y.~Brechet.
\newblock The physical origin of the coffin-manson law in low-cycle fatigue.
\newblock {\em Europhys. Lett.}, pages 433--438, 1992.

\bibitem{Watanabe}
S.~Watanabe.
\newblock On discontinuous additive functionals and lv\'ey measures of a markov
  process.
\newblock {\em Japan. J. Math.}, 34, 1964.

\bibitem{Wei}
E.~W. Weibull.
\newblock A statistical theory of the strength of materials.
\newblock {\em Ingeniors Vetenskaps Akad. Handl.}, 151:1--45, 1939.

\end{thebibliography}
\end{document}